\documentclass[12pt,reqno]{amsart}
\usepackage{amsmath,amsfonts,amsthm,amsopn,cite,amssymb}
\numberwithin{equation}{section}
\usepackage{color,url,setspace,wasysym}
\usepackage{lscape}
\usepackage[papersize={209mm, 296mm},
inner=24mm, outer=24mm,
top=28mm, bottom=28mm,
headsep=10mm,footskip=12mm]{geometry}
\usepackage[unicode,pdfencoding=unicode,pdfusetitle,
bookmarks=true,bookmarksnumbered=false,bookmarksopen=false,
breaklinks=false,pdfborder={0 0 1},backref=false,colorlinks=false]{hyperref}

\allowdisplaybreaks[1]

\usepackage{mathrsfs}
\usepackage{enumitem,color}

\theoremstyle{plain}
\newtheorem*{theorem_A}{Theorem A}
\newtheorem*{theorem_B}{Theorem B}
\newtheorem{theorem}{Theorem}[section]
\newtheorem{prop}[theorem]{Proposition}

\newtheorem{lemma}[theorem]{Lemma}

\theoremstyle{definition}

\theoremstyle{remark}

\theoremstyle{remark}
\newtheorem*{rem*}{Remark}
\newtheorem{remark}{Remark}[section]
\title[Multiple sampling and interpolation]{Multiple sampling and interpolation in weighted Fock spaces of entire functions}
\author{}
\date{}
\newcommand{\numberofpoints}[4]{n\left( #1,#2,#3,m_{#4}\right)}
\newcommand{\density}{1/\pi}
\newcommand{\ddensity}{\frac{1}{\pi}}

\newcommand{\fock}[1]{\ensuremath{\mathcal{F}_{#1}^{2}}}

\newcommand{\pairsetmult}[1]{(#1,\,m_{#1})}

\newcommand{\abs}[1]{\ensuremath{\left| #1 \right| }}

\newcommand{\norm}[1]{\lVert#1\rVert}
\newcommand{\di}{i}

\newcommand{\dbarstar}[2]{
	\ifnum\pdfstrcmp{#2}{1}=0
	{\bar{\partial}^{*}\hspace{0.07em}}\hspace{-0.2em}{#1}
	\else
	{\bar{\partial}^{*^{\hspace{0.05em}(#2)}}}\hspace{-0.2em}{#1}	\fi
}
\newcommand{\dbarstarr}[3]{
	\ifnum\pdfstrcmp{#2}{1}=0
	{\bar{\partial}_{#3}^{*}\hspace{0.07em}}\hspace{-0.0em}{#1}
	\else
	{\bar{\partial}_{#3}^{*^{\hspace{0.05em}(#2)}}}\hspace{-0.2em}{#1}	\fi
}
\newcommand\restr[2]{{%
		\left.\kern-\nulldelimiterspace %
		#1 %
		\vphantom{\big|} %
		\right|_{#2} %
}}

\newcommand{\field}[1]{\mathbb{#1}}
\newcommand{\bR}{\field{R}}
\newcommand{\bN}{\field{N}}

\newcommand{\bC}{\field{C}}
\newcommand{\dA}{dA}
\newcommand{{\nLambda}}{n_\Lambda}

\def\Xint#1{\mathchoice
	{\XXint\displaystyle\textstyle{#1}}%
	{\XXint\textstyle\scriptstyle{#1}}%
	{\XXint\scriptstyle\scriptscriptstyle{#1}}%
	{\XXint\scriptscriptstyle\scriptscriptstyle{#1}}%
	\!\int}
\def\XXint#1#2#3{{\setbox0=\hbox{$#1{#2#3}{\int}$}
		\vcenter{\hbox{$#2#3$}}\kern-.5\wd0}}

\def\dashint{\Xint-}

\author[L. A. Escudero]{Luis Alberto Escudero}
\address{Acoustics Research Institute, Austrian Academy of Sciences,
Wohllebengasse 12-14 A-1040, Vienna, Austria}
\email{lescudero@kfs.oeaw.ac.at}

\author[A. Haimi]{Antti Haimi}
\address{Acoustics Research Institute, Austrian Academy of Sciences,
Wohllebengasse 12-14 A-1040, Vienna, Austria}
\email{ahaimi@kfs.oeaw.ac.at}

\author[J. L. Romero]{Jos\'{e} Luis Romero}

\address{Faculty of Mathematics,
University of Vienna,
Oskar-Morgenstern-Platz 1,
A-1090 Vienna, Austria\\and\\
Acoustics Research Institute, Austrian Academy of Sciences,
Wohllebengasse 12-14 A-1040, Vienna, Austria}
\email{jose.luis.romero@univie.ac.at,
jlromero@kfs.oeaw.ac.at}

\thanks{L. A. E. and J. L. R. gratefully acknowledge support from the Austrian Science Fund (FWF): P 29462 - N35 and
Y 1199, while A. H. was supported by the Austrian Science Fund (FWF): P 31153-N35.
A. H. and J. L. R. also acknowledge support from the WWTF grant INSIGHT (MA16-053).
}

\begin{document}
\begin{abstract}
We characterize sampling and interpolating sets with derivatives in weighted Fock spaces on the complex plane in terms of their weighted Beurling densities.
\end{abstract}
\maketitle

\section{Introduction}

\subsection{Results and context}

Let $\phi:\bC\to\bR$ be a subharmonic function with Laplacian bounded above and below by positive constants. The \emph{weighted Fock space} of entire functions is
$$
\fock{\phi} :=  \left\{f:\bC\to\bC : f \text { holomorphic, } \norm{f}_\phi^2 :=\int_{\mathbb{C}}|f(z)|^{2} e^{-\phi(z)} \dA(z)<\infty\right\},
$$
where $\dA$ denotes the Lebesgue measure on $\bC$. We study sampling and interpolation on weighted Fock spaces, where we sample or interpolate using not only the function values, but also the values of its derivatives. Such process is sometimes called
multiple or Hermite sampling and interpolation. While sampling and interpolation theory provide a mathematical foundation for the tasks of digitalization and encoding of analog signals into bit-streams, the use of derivatives incorporates
the \emph{trend} of the signal as well, and is well studied in Paley-Wiener spaces and shift-invariant spaces on the real line; see, e.g., \cite{Raz95, AGH17,SR16,GRS20}.

In this article, we characterize the sets that allow for multiple sampling and interpolation in weighted Fock spaces in terms of certain densities, provided that the number of derivatives considered at each sampling point is bounded. The precise assumptions and definitions are given in Section \ref{sec:definitions_and_results}.

\medskip

The prime example of a Fock space is the \emph{Bargmann-Fock} space, where $\phi(z)=\alpha |z|^2$, and $\alpha > 0$ \cite{MR2934601}. For this weight,
the Heisenberg group acts irreducibly on $\fock{\phi}$, which results in a important homogeneity property called \emph{translation invariance}
\cite[Section 2.6]{MR2934601}. More precisely, $\fock{\phi}$ is invariant under the so-called \emph{Bargmann-Fock} shifts \cite{MR2934601}:
\begin{align}\label{eq_intro_shift}
f(z) \mapsto e^{-\tfrac{\alpha}{2} |w|^2 + \alpha z \bar{w}} f(z-w),
\qquad w \in \mathbb{C}.
\end{align}
Necessary and sufficient conditions for sampling and interpolation without derivatives are fully described in terms of planar Beurling densities \cite{seip1, seip2}. Specifically, the lower and upper Beurling densities 
of a set $\Lambda \subseteq \mathbb{C}$ are
\begin{align}\label{eq_into_B1}
D^-(\Lambda) &:= \liminf_{R \longrightarrow \infty}
\inf_{z \in \mathbb{C}} \frac{\# \left( \Lambda \cap B(z,R) \right)}
{\pi R^2},\\
\label{eq_into_B2}
D^+(\Lambda) &:= \limsup_{R \longrightarrow \infty}
\sup_{z \in \mathbb{C}} \frac{\# \left( \Lambda \cap B(z,R) \right)}
{\pi R^2},
\end{align}
where $B(z,R)$ denotes the Euclidean ball with center $z$ and radius $R$. For
$\phi(z)=\alpha |z|^2$ and a \emph{separated} set $\Lambda$, the existence of two sampling constants $A,B>0$ leading to a sampling inequality
\begin{align}\label{eq_intro_samp}
A\|f\|_{\phi}^{2} \leq \sum_{\lambda \in \Lambda}   \left|{f(\lambda)}\right|^{2}  e^{-\phi(\lambda)} \leq B\|f\|_{\phi}^{2}, \qquad f \in \fock{\phi},
\end{align}
is completely characterized by the density condition 
$D^-(\Lambda) > \frac{\alpha}{\pi}$. Similarly, the density condition
$D^+(\Lambda) < \frac{\alpha}{\pi}$ completely characterizes the validity of the following interpolation property: given a sequence $\{c_\lambda: \lambda \in \Lambda\} \subset \mathbb{C}$ satisfying the growth condition
$\sum_{\lambda \in \Lambda} |c_\lambda|^2 e^{- \phi(\lambda)} < \infty$
there exists (at least) one function $f \in \fock{\phi}$ such that
\begin{align}\label{eq_intro_int2}
f(\lambda)=c_\lambda, \qquad \lambda \in \Lambda.
\end{align}
Slightly more technical formulations of these density characterizations apply to arbitrary, possibly non-separated sets \cite{seip1, seip2}. 
These results parallel Beurling's sampling and interpolation theorems for the Paley-Wiener space
of square integrable functions on the real line having Fourier transforms supported on the unit interval, with the difference that the latter results do not provide a complete characterization in terms of densities, as the necessary density condition involves a non-strict inequality while the sufficient involves a strict one \cite{be66,be89} \footnote{See also \cite{landau67} for necessary density conditions for sampling and interpolation in Paley-Wiener spaces with arbitrary spectra.}. Translation invariance plays a key role in \cite{seip1, seip2}, and the arguments do not seem to extend easily to more general weights. 

For the classical weight $\phi(z)=\alpha |z|^2$, sampling and interpolation with bounded multiplicities are studied in \cite{seip3}, while unbounded multiplicities are treated in \cite{MR3558232
}. In this setting the sampling \eqref{eq_intro_samp}
and interpolation equations \eqref{eq_intro_int2} involve not only the function $f$ but also its \emph{translation covariant derivative}
\begin{align}\label{eq_intro_dera}
f(z) \mapsto \alpha \bar{z}f(z) - \partial f(z)
\end{align}
applied iteratively. The operator \eqref{eq_intro_dera} is called covariant, because it commutes with the Bargmann-Fock shifts \eqref{eq_intro_shift}. When considering multiplicities, the characterization of sampling and interpolation is in terms of weighted variants of the planar Beurling densities \eqref{eq_into_B1}, \eqref{eq_into_B2}, where each point $\lambda$ is counted multiply, according to the number of derivatives that are evaluated at $\lambda$ \cite{seip3}. (See Section \ref{sec_is} for more details.)

Sampling and interpolation on Fock spaces with general weights has been studied in \cite{berndtsson-ortega-95} and \cite{ortega-seip-98}. The characterization is in terms of variants of the Beurling densities \eqref{eq_into_B1}, \eqref{eq_into_B2}, where
the factor $\pi R^2$ is replaced by the weighted measure
\begin{align}\label{eq_intro_factor}
\int_{B(z,R)} \Delta \phi \,dA = \int_{B(z,R)} \partial \bar\partial \phi \,dA.
\end{align}
As $\partial \bar \partial \big[ \alpha |z|^2 \big] = \alpha$, Beurling densities defined using \eqref{eq_intro_factor} extend the classical ones \eqref{eq_into_B1}, \eqref{eq_into_B2} up to normalization. For general weights, the lack of translation invariance demands new techniques, such as the introduction of \emph{approximate translation operators} \cite{ortega-seip-98}, 
and $\bar{\partial}$-surgery together with Riesz decompositions \cite{berndtsson-ortega-95} (see Section \ref{sec:preliminaries}).

In this article we simultaneously extend the results from \cite{seip3, berndtsson-ortega-95, ortega-seip-98} to study sampling and interpolation with multiplicities on Fock spaces with general weights.
Our starting point is the observation that, for the classical weight $\phi(z)=\alpha|z|^2$, the covariant derivative \eqref{eq_intro_dera} is the formal adjoint of the $\bar{\partial}$-operator with respect to the scalar product that induces the norm of $\fock{\phi}$. Explicitly,
\begin{align}\label{eq_intro_der}
\dbarstarr{f}{1}{\phi}:=-e^{\phi} \partial\left(e^{-\phi} f \right).
\end{align}
We consider this operator for general weights $\phi$, and study
sampling and interpolation involving iterated applications of 
$\bar{\partial}^*_\phi$. We derive a characterization of these properties in terms of a suitable variant of the Beurling densities \eqref{eq_into_B1}, \eqref{eq_into_B2}, that counts points multiply according to the number of concerned derivatives as in \cite{seip3}, and weights balls with the Laplacian factor \eqref{eq_intro_factor} as in \cite{berndtsson-ortega-95, ortega-seip-98}.

\subsection{Motivation}
One main motivation for our results is validating the differential operator \eqref{eq_intro_der} as an adequate replacement for the covariant derivative
\eqref{eq_intro_dera}, by showing that it leads to similar sampling and interpolation properties. Other candidates for this role are the differential operators defined as ordinary differentiation at the origin, conjugated with the approximate translation operators from \cite{ortega-seip-98}. This second choice is however unnatural, as it arbitrarily distinguishes the origin, and leads to complicated compatibility issues (stemming from the fact that approximate translation operators are not associated with a group representation).
On the other hand, although they lack the simplicity of \eqref{eq_intro_der}, differential operators associated with approximate translations would allow a more straightforward generalization of the proofs of \cite{berndtsson-ortega-95, ortega-seip-98}. In any case, one can use our present results to show a posteriori that both differential operators - \eqref{eq_intro_der} and the ones defined in terms of approximate translations - do lead to similar sampling and interpolation theorems.

Fock spaces find important applications in the simultaneous time-frequency analysis of real variable functions. The most successful applications concern the classical weights and \emph{Gabor systems}, which are structured functional dictionaries for $L^2(\mathbb{R})$,
\begin{align}\label{eq_intro_gabor}
\mathcal{G}(g,\Lambda)
=\big\{ e^{2 \pi i b \cdot} g(\cdot - a) : \lambda=(a,b) \in \Lambda \big\}
\end{align}
generated by the Gaussian function $g(t) = e^{-\pi t^2}$, $t \in \mathbb{R}$, and their \emph{time-frequency shifts} along a given set of nodes $\Lambda \subseteq \mathbb{R}^2$. By means of the \emph{Bargmann transform}, the spanning properties of \eqref{eq_intro_gabor} can be reformulated as sampling and interpolation properties of $\Lambda$ on the Fock space with weight $\phi(z)=\pi|z|^2$ \cite{dagr88}, and thus characterized in terms of Beurling densities \cite{seip1,seip2, lyub92}. Multiple sampling and interpolation on $\fock{\phi}$ concerns in turn the spanning properties of \emph{multi-window Gabor systems}, where the Gaussian function is supplemented with additional Hermite functions \cite[Remark 1]{seip3}. Other weights $\phi(z)=\alpha\pi|z|^2$ allow a similar analysis for
the general Gaussian function $g(t) = e^{-a t^2}$. Applications of more general weighted Fock spaces to time-frequency analysis are more recent, and mainly connected to non-linear properties of the so-called short-time Fourier transform, where for example weights are chosen in a signal dependent fashion, see e.g., \cite{GR}. We expect our current results to find applications in this direction.

\subsection{Organization and technical overview}
In Section \ref{sec:definitions_and_results} we present the precise definitions and assumptions as well as the main results.
Section \ref{sec:preliminaries} provides useful estimates from potential theory, which are important to treat derivatives of functions in Fock spaces. The sufficiency of the density conditions for sampling and interpolation is shown under additional technical assumptions in Section \ref{sec_suf}, following closely the arguments from \cite{berndtsson-ortega-95}. The most important task here is showing that the operator \eqref{eq_intro_der} is suitably compatible with that line of argument, which is based on H\"ormander's $L^2$-estimates for $\overline{\partial}$\cite{Hormander}. The necessity of the density conditions for multiple sampling and interpolation in the weighted case is challenging, as the classical approach relies on translation invariance \cite{seip3}, and
the one in \cite{ortega-seip-98} on a subtle form of approximate translation invariance. Instead of adapting the delicate proof from \cite{ortega-seip-98}, in Section \ref{sec_nec} we develop a perturbation argument that allows us to reduce the problem to the case of no multiplicities. This line of reasoning seems to provide a more direct proof of the necessary density conditions for multiple sampling and interpolation even for the classical weights \cite{seip3}. The main results are finally proved in Section \ref{sec:proof_main_result}, where the remaining technical assumptions are removed. Section \ref{sec:Bell} is an appendix listing auxiliary results related to iterated derivatives of composite functions.

\section{Definitions and main results}\label{sec:definitions_and_results}
\subsection{Sets with multiplicity}

A \emph{set with multiplicities} is a pair $\left(\Lambda, m_{\Lambda}\right)$,
where $\Lambda \subseteq \bC$, and $m_{\Lambda} : \Lambda \rightarrow \mathbb{N}$
is a bounded function called \emph{multiplicity function}. In problems concerning evaluations of functions at $\Lambda$, the number $m_{\Lambda}(\lambda)$ indicates how many derivatives are involved at the point $z=\lambda$. More precisely, $m_\Lambda(\lambda)=k$ indicates the interpolation or sampling of a function and its first $k-1$ derivatives at $\lambda$.

\subsection{Notation} Complex disks are denoted $B(z,r):=\{y\in\bC : |z-y|< r \}$,
with $z\in\bC$ and $r>0$. We use the notation $x\lesssim y$ if there exists a constant such that $x\leq C y$. The constant is normally allowed to depend on $\sup_{\lambda \in \Lambda} m_{\Lambda}(\lambda)$ and $\phi$, while other dependencies are noted explicitly. The precise value of unspecified constants may vary from line to line.

For a set with multiplicities $(\Lambda, m_{\Lambda})$, we denote by $\ell_\phi^2(\Lambda, m_\Lambda)$ the
space of sequences $a \equiv \big(a_{(\lambda,j)} \big)_{\lambda\in\Lambda,\, 0\leq j \leq m_{\Lambda}(\lambda)-1 }$
of complex numbers with finite norm:
$$\norm{a}_{\ell_\phi^2(\Lambda,m_\Lambda)}^2 = \sum_{\lambda\in\Lambda} \sum_{j=0}^{m_{\Lambda}(\lambda)-1} |a_{(\lambda,j)}|^2 e^{-\phi(\lambda)}.$$
When $m_\Lambda \equiv 1$ we simply write $\ell_\phi^2(\Lambda)$.

\subsection{Separated sets}
A set $\Lambda \subseteq \bC$ is called \emph{relatively separated} if
\begin{equation*}\operatorname{rel}(\Lambda):=\sup \left\{\#\left(\Lambda \cap B(z,1)\right): z \in \mathbb{C}\right\}<\infty,\end{equation*}
and it is called \emph{(uniformly) separated} if
\begin{equation*}\rho(\Lambda):=\inf \left\{\left|\lambda-\lambda^{\prime}\right|: \lambda \neq \lambda^{\prime} \in \Lambda\right\}>0.
\end{equation*}
Separated sets are relatively separated, and relatively separated sets are finite unions of separated sets.
For simplicity of notation, we write $\rho$ instead of $\rho(\Lambda)$ when no confusion can arise.
We say that a set with multiplicities $(\Lambda, m_\Lambda)$ is separated (respectively relatively separated) if $\Lambda$ is separated (respectively relatively separated).

\subsection{Wirtinger derivatives and the $\dbarstar{}{1}$ operator}\label{sec_w}

We normalize the Laplacian as
\begin{equation*}
\Delta:=\frac{1}{4}\left(\partial_{x}^{2}+\partial_{y}^{2}\right)=\partial \bar{\partial},
\end{equation*}
where $\partial$ and $\bar{\partial}$ denote the Wirtinger derivatives:
\[
\partial=\frac{1}{2}\left(\partial_{x}-\di \partial_{y}\right) \quad \text { and } \quad \bar{\partial}=\frac{1}{2}\left(\partial_{x}+\di \partial_{y}\right).
\]
For $f\in C^{1}(\bC)$ we define the operator
$$\dbarstarr{f}{1}{\phi}:=-e^{\phi} \partial\left(e^{-\phi} f \right),$$
which is the formal adjoint of the $\bar{\partial}$-operator with respect to the weighted scalar product
$$\langle f, g\rangle_{\phi}=\int f(z) \overline{g(z)} e^{-\phi(z)} \dA(z).$$
When the dependence on $\phi$ is clear from the context, we write for short $\dbarstar{}{1}$. For the classical weight $\phi(z)=\alpha|z|^2$,
$\dbarstarr{f}{1}{\phi}$ is the covariant derivative \eqref{eq_intro_dera}. 

\subsection{Interpolating sets and sampling sets}\label{sec_is}
We say that $\left(\Lambda, m_{\Lambda}\right)$ is an\emph{ interpolating set }for $\fock{\phi}$, if for any sequence
$c\in\ell_\phi^2{\pairsetmult{\Lambda}}$
there exists a function $f \in \fock{\phi} $  such that $\dbarstar{f(\lambda)}{j}=c_{(\lambda, j)}$, for all $\lambda\in\Lambda$ and  $j\in\{0,\ldots,{m_{\Lambda}(\lambda)-1}\}$.
We say that $\left(\Lambda, m_{\Lambda}\right)$ is a \emph{sampling set} for $\fock{\phi}$, if there
exist constants $A, B>0$ such that
\begin{equation*}
A\|f\|_{\phi}^{2} \leq \sum_{\lambda \in \Lambda}   \sum_{j=0}^{m_{\Lambda}(\lambda)-1} \left|\dbarstar{f(\lambda)}{j}\right|^{2}  e^{-\phi(\lambda)} \leq B\|f\|_{\phi}^{2}.
\end{equation*}
For the classical weight $\phi(z)=\alpha|z|^2$, using the covariance property of $\overline{\partial}^*_{\phi}$, is easy to see that
\begin{align*}
e^{-\tfrac{1}{2}\phi(\lambda)} \cdot \dbarstar{f(\lambda)}{j}
= (-1)^j \langle f, W_\lambda e_j \rangle_\phi,
\end{align*}
where $e_j(z)=\left(\alpha^{j+1}  / {\pi}\right)  z^j$ and $W_w$ denotes the Bargmann-Fock shift by $w$ \eqref{eq_intro_shift}. 
Multiple sampling and interpolating sets with respect to the weight $\phi(z)=\alpha|z|^2$ are defined in \cite{seip3} in terms of Bargmann-Fock shifts of the normalized monomials $f_j(z) = (\alpha^j / j!)^{1/2} z^j$. Since, as in \cite{seip3}, we only consider sets with bounded multiplicity, for the classical weight $\phi(z)=\alpha|z|^2$, the present notion of multiple sampling and interpolation is equivalent to the one from \cite{seip3}.

\subsection{Beurling densities}

Given a set with multiplicities $\left(\Lambda, m_{\Lambda}\right)$, $\numberofpoints{z}{r}{\Lambda}{\Lambda}$ is defined as the number of points of the set $\Lambda$ in the disk $B(z,r)$, counted with multiplicities. That is,
$$
\numberofpoints{z}{r}{\Lambda}{\Lambda}
 = \nu \star  1_{B(0, r)} (z),
$$ where
\begin{align*}
\nu = \nu_\Lambda :=\sum_{\lambda \in \Lambda} m_\Lambda (\lambda) \cdot  \delta_{\lambda}
\end{align*}
is a weighted sum of Dirac deltas, and $1_{B(0, r)}$ is the characteristic function of the disk.
The lower and upper Beurling densities of $\left(\Lambda, m_{\Lambda}\right)$ with respect to $\phi$ are:
\begin{align}\label{lower_density}
D_{\phi}^{-}\left(\Lambda, m_{\Lambda}\right)&=\liminf _{r \rightarrow \infty} \inf _{z \in \mathbb{C}} \frac{ \numberofpoints{z}{r}{\Lambda}{\Lambda} }{\int_{B(z, r)} \Delta \phi\, \dA},
\\
\label{upper_density}
D_{\phi}^{+}\left(\Lambda, m_{\Lambda}\right)&=\limsup _{r \rightarrow \infty} \sup _{z \in \mathbb{C}} \frac{   \numberofpoints{z}{r}{\Lambda}{\Lambda}}{\int_{B(z, r)} \Delta \phi\, \dA}.
\end{align}

\subsection{Compatibility assumptions}\label{sec:assumptions}
We always assume that sets with multiplicities have bounded multiplicity functions and denote
\begin{align*}
{\nLambda} := \sup_{\lambda \in \Lambda} m_{\Lambda}(\lambda) -1.
\end{align*}
We say that $(\Lambda, m_\Lambda)$ and $\phi$ are \emph{compatible} if the following conditions hold. For ${\nLambda}=0$ or ${\nLambda}=1$, we only require that $m<\Delta\phi<M$ for some positive constants $m$, $M$.
For ${\nLambda}\geq 2$, we additionally require $\phi \in C^{{\nLambda}+1}(\mathbb{C})$ and  $\sup_{z\in\bC} |\partial_z^{j} \Delta\phi(z) |< \infty$, for all $j=1, \ldots,{\nLambda}-1$.

The conditions are chosen so that there exists $\varepsilon_r>0$ for which  \begin{equation}\label{eq:bound_potential} \left| \partial_z^j\left( \int_{B(\zeta,r)}\log|w-z|\ \Delta\phi(w)\, \dA(w)\right) \right| \leq C_r,\end{equation}
for every $\zeta \in\bC$, $z\in B(\zeta, \varepsilon_r)$ and $0\leq j \leq {\nLambda}$.

\subsection{Main results} Our goal is to derive the following characterization of sampling and interpolating sets.

\begin{theorem_A} Let $(\Lambda, m_\Lambda)$ be a set with multiplicities that is compatible with the weight $\phi$.
Then $(\Lambda, m_\Lambda)$ is interpolating for $\fock{\phi}$ if and only if it is separated and satisfies $D_{\phi}^{+}(\Lambda, m_\Lambda)<\density{}$.
\end{theorem_A}
\begin{theorem_B} Let $(\Lambda, m_\Lambda)$ be a set with multiplicities that is compatible with the weight $\phi$.
Then $(\Lambda, m_\Lambda)$ is sampling for $\fock{\phi}$ if and only if it is relatively separated and $\Lambda$ contains a separated subset $\Lambda^\prime$ satisfying $D_{\phi}^{-}\left(\Lambda^{\prime}, \restr{m_\Lambda}{\Lambda^\prime}\right)>\density{}$.
\end{theorem_B}

\section{Riesz decomposition}\label{sec:preliminaries}

The Riesz decomposition, presented below, describes the structure of subharmonic functions on bounded subdomains of the complex plane. It is an essential tool to study weighted Fock spaces, as their very definition involves a subharmonic function.

\begin{theorem}[Riesz decomposition]\label{thm:Riesz_decomp}
	Let $D$ be a domain in $\mathbb{C}$, and let $u:D\rightarrow \mathbb{R}$ be a subharmonic function. If $K\subseteq D$ is an open and relatively compact set, then there exists a harmonic function $h$ on $K$ such that
	\begin{equation*}
	u(z)=h (z)+G[\Delta u](z), \qquad z \in K,
	\end{equation*}
	where \begin{equation}\label{eq:green_potential}G[\Delta u](z) =\frac{2}{\pi} \int_{K}\log|w-z|\ \Delta u(w)\, \dA(w).\end{equation}
\end{theorem}

The function $G[\Delta u]$ in \eqref{eq:green_potential} is called \emph{the logarithmic potential} of $\Delta u$ on the domain $K$. For more on Riesz decomposition, see \cite{MR1334766, MR0460672, MR1485778}.

We recall the following fact from potential theory, (see, e.g., \cite[Thm. 10]{XinweiYu2011}).

\begin{lemma}\label{lemma_pde}
			Let $\Omega \subset \mathbb{R}^{n}$ be open and bounded. Let $B_{1}=B\left(x_{0}, R\right)$, $B_{2}=B\left(x_{0}, 2R\right) \subset \bar{B_2} \subset \Omega$ be concentric balls.
Suppose that $u$ solves $\Delta u=f$ in $\overline{B_{2}}$, and $f \in C^{0}(\Omega)$, then $u \in C^{1, \alpha}(B_1)$ for any $\alpha \in(0,1),$ and
			$$
			\|u\|_{C^{1, \alpha}\left(B_1\right)} \lesssim \|f\|_{L^\infty(B_2)}+\|u\|_{L^{2}(B_2)}.
			$$

			Here, $$\norm{u}_{C^{k,\alpha}(\Omega)}:=\sum_{0\leq \abs{\beta} \leq k} \norm{\partial^\beta u}_{L^\infty(\Omega)} + \sup_{\abs{\beta}=k} \sup_{\substack{x,y\in\Omega \\ x\neq y}} \frac{\abs{D^\beta u (x)-D^\beta u (y)}}{\abs{x-y}^\alpha}.$$
		\end{lemma}

		As an application, we have the following bounds for Riesz decomposition.

\begin{lemma}\label{lemma:bounds_potentials}
Let $(\Lambda, m_\Lambda)$ and $\phi$ be compatible, and let $\varepsilon > 0$.
For each $\lambda \in \Lambda$, let
	\begin{equation*}
\phi(z)=h_\lambda(z)+G[\Delta \phi](z), \qquad z \in B(\lambda,\varepsilon),
\end{equation*}
be the Riesz decomposition of $\phi$ on $B(\lambda,\varepsilon)$, where $h_\lambda$ is harmonic and $G[\Delta \phi]$ is the logarithmic potential of $\Delta \phi$ on $B(\lambda,\varepsilon)$.  Write $h_\lambda = 2 \operatorname{Re} H_\lambda$, where $H_\lambda$ is holomorphic, and set $G_{\lambda}=H_\lambda-H_\lambda(\lambda)$.

Then there exists $C_\varepsilon > 0$, such that
for $z\in B(\lambda,\varepsilon)$,
\begin{align}
\left| G[\Delta\phi](z) \right|&\leq C_\varepsilon, \\
\left| \phi(z)-\phi(\lambda)-2\operatorname{Re} G_{\lambda}(z)\right|&\leq C_\varepsilon,
\end{align}
and if $z\in B(\lambda,\varepsilon/{2^{n_\Lambda}})$ and $1\leq k\leq {\nLambda}$,
\begin{align}\label{eq_fff}
|\partial^{k} G[\Delta\phi](z)| &\leq C_\varepsilon,
\\\label{eq_ffff}
\left| \partial^k(G_{\lambda} - \phi)(z)        \right| &\leq C_\varepsilon.
\end{align}
The constant $C_\varepsilon$ depends on $\varepsilon$ but not on $\lambda$.
\end{lemma}
\begin{proof}
Let $z\in B(\lambda, \varepsilon)$. We first note that
	\begin{align*}\big|G[\Delta u](z)\big| &\leq \frac{2}{\pi} \int_{B(\lambda,\varepsilon)}\big|\log|w-z|\big|\ \big|\Delta u(w)\big|\, \dA(w) \\ &\leq \frac{2}{\pi} \sup_{w\in\bC} |\Delta\phi(w)|  \int_{B(0,2\varepsilon)}\big|\log|w|\big|\ \, \dA(w) =: C_\varepsilon^{(1)}.
	\end{align*}
Note also that
	\begin{equation}\label{eq:def_g_lambda}
\phi(z)-\phi(\lambda)-2\operatorname{Re} G_{\lambda}(z) = G[\Delta \phi](z)-G[\Delta \phi](\lambda).
	\end{equation}
Consequently,
	\begin{equation}
	\left|\phi(z)-\phi(\lambda)-2\operatorname{Re} G_{\lambda}(z)\right|=\left|G[\Delta \phi](z)-G[\Delta \phi](\lambda)\right|< 2  C_\varepsilon^{(1)}.
	\end{equation}

We now show similar bounds for the derivatives of the logarithmic potential of $\Delta\phi$. Suppose $n_\Lambda\geq 1$.
As shown in \cite[Lemma 4.1]{MR1814364}, since $\Delta\phi$ is bounded and integrable in $B(\lambda,\varepsilon/2)$, $G[\Delta\phi]\in C^1(\lambda,\varepsilon/2)$. Moreover, as consequence of \cite[Eq. 4.8]{MR1814364} for any $z\in B(\lambda,\varepsilon/2)$,
\begin{equation}\label{eq:first_derivative_potential_bound}\left|\partial G[\Delta\phi](z) \right| \leq \frac{2}{\pi} \sup_{w\in\bC}|{\Delta\phi}(w)| \int_{B(0,2\varepsilon)}     \left|{w} \right|^{-1} \dA(w) =C_\varepsilon^{(2)}.\end{equation}

Suppose now that $n_\Lambda \geq 2$. Recall that in this case, $\phi$ is assumed to be $C^{n_{\Lambda}+1}$, thus the Laplacian is $C^{n_{\Lambda}-1}$.
By Theorem \ref{thm:Riesz_decomp}, \begin{equation}\label{eq:g_delta} \Delta\left( G[\Delta\phi]\right)=\Delta\phi,\end{equation}
on $B(\lambda,\varepsilon)$. Taking derivatives in \eqref{eq:g_delta} we obtain
$\Delta \partial G[\Delta\phi]=\partial \Delta\phi$. On account of Lemma \ref{lemma_pde}, since $\partial\Delta\phi\in C(B(\lambda,\varepsilon))$, we obtain that $\partial G[\Delta\phi] \in C^1(B(\lambda,\varepsilon/4))$ and that there exists a constant $\tilde{C_2}>0$ for which
$$\norm{\partial^2 G[\Delta\phi]}_{L^\infty (B(\lambda,\varepsilon/4))} \leq \tilde{C_2}
\left(\norm{\partial  \Delta\phi}_{L^\infty(   B(\lambda,\varepsilon/2) )}+\norm{\partial G[\Delta\phi]}_{L^{2}(B(\lambda,\varepsilon/2))}\right),
$$ where the right-hand of the last inequality is uniformly bounded as a consequence of  \eqref{eq:first_derivative_potential_bound} and the assumption that $\partial\Delta\phi$ is uniformly bounded.
Iterating this argument, we obtain that
$$\norm{\partial^k G[\Delta\phi]}_{L^\infty (B(\lambda,\varepsilon/2^k))} \leq \tilde{C_k}
\left(\norm{\partial^{k-1}  \Delta\phi}_{L^\infty(   B(\lambda,\varepsilon/2^{k-1}) )}+\norm{\partial^{k-1} G[\Delta\phi]}_{L^{2}(B(\lambda,\varepsilon/2^{k-1}))}\right),
$$
where the right-hand is uniformly bounded. This proves \eqref{eq_fff}.

Finally, since $G$ is holomorphic, repeated differentiation of \eqref{eq:def_g_lambda} yields,
\begin{equation}\label{eq:relationship_derivatives}
\partial^k G_{\lambda} (z) - \partial^k\phi (z) =
\partial^k [ 2 \operatorname{Re} G_{\lambda}] (z) - \partial^k\phi (z)
= - \partial^k G[\Delta \phi] (z).
\end{equation}
Thus \eqref{eq_fff}, gives \eqref{eq_ffff}.
\end{proof}

The following lemma helps profit from Riesz decomposition in estimates involving derivatives.

\begin{lemma}\label{lemma:bound_partial_f_eh}
Let $(\Lambda, m_\Lambda)$ and $\phi$ be compatible, $D\subseteq\bC$ a domain, $K \subseteq D$ open and relatively compact, and $f,H:D\longrightarrow \bC$ analytic functions. Let $G[\Delta\phi]$ be as in Theorem \ref{thm:Riesz_decomp}. Then there exist a constant $C=C_{\phi, K}$, such that for every $\lambda \in \Lambda \cap K$,
and $0 \leq j \leq n_\Lambda$,
the following estimates hold:
	\begin{enumerate}\item[(i)]	$\displaystyle	|\partial^j(fe^{-H})(\lambda) |^2 \leq C \sum_{k=0}^j |\partial^k ( fe^{-H-G[\Delta\phi]}) (\lambda)  |^2  ,$

		\item[(ii)] $\displaystyle	|\partial^j(fe^{-H-G[\Delta\phi]})(\lambda) |^2 \leq C \sum_{k=0}^j |\partial^k ( fe^{-H}) (\lambda)  |^2  .$
	\end{enumerate}
\end{lemma}
\begin{proof}
By Leibniz rule:
	\begin{align*}
	\partial^j(fe^{-H})=\partial^j(fe^{-H-G[\Delta\phi]+G[\Delta\phi]}) &= \sum_{k=0}^{j} \binom{j}{k}  \partial^{k} (f e^{-H-G[\Delta\phi]}) \partial^{j-k}(e^{G[\Delta\phi]}).
	\end{align*}
	By Lemma \ref{lemma:bounds_potentials}, $\partial^{j-k}(e^{G[\Delta\phi]})$ is bounded for $0\leq k\leq j$. This yields (i); (ii) follows similarly.
\end{proof}

Finally, we note that Cauchy bounds extend to weighted derivatives.

\begin{lemma}\label{lemma:compare_dbarstar_normf}
Let $\left(\Lambda, m_{\Lambda}\right)$ be a separated set with multiplicities that is compatible with $\phi$.
Then for each $\lambda\in\Lambda$ and $j \leq {\nLambda}$,
	$$\left|\dbarstar{f}{j}(\lambda)\right|^2 e^{-\phi(\lambda)}\lesssim \int_{B(\lambda,1)}  |f(w)|^2 e^{-\phi(w)}\ dA(z).$$
\end{lemma}
\begin{proof}
	For each $\lambda\in\Lambda$ we write $\phi(z)=h_\lambda(z)+G[\Delta\phi](z)$ as in Theorem \ref{thm:Riesz_decomp} where the decomposition is taken on the set $B(\lambda,1)$. We apply Lemma \ref{lemma:bound_partial_f_eh}, together with Cauchy's bound to obtain
	\begin{align*}
	\left|\dbarstar{f}{j}(\lambda)\right|^2 e^{-\phi(\lambda)}&=\left|\partial^j\left(fe^{-\phi}\right)(\lambda)e^{\phi(\lambda)}\right|^2 e^{-\phi(\lambda)} =\left|\partial^j\left(fe^{-H_\lambda-\overline{H_\lambda}-G[\Delta\phi]}\right)(\lambda)\right|^2 e^{\phi(\lambda)}\\
	&=\left|\partial^j\left(fe^{-H_\lambda-G[\Delta\phi]}\right)(\lambda)\right|^2 \left| e^{-\overline{H_\lambda}} \right|^2 e^{\phi(\lambda)} =\left|\partial^j\left(fe^{-H_\lambda-G[\Delta\phi]}\right)(\lambda)\right|^2 e^{\phi(\lambda)-h_\lambda(\lambda)}\\
	&\lesssim \sum_{k=0}^j \left|\partial^k ( fe^{-H_\lambda}) (\lambda)  \right|^2 \lesssim \int_{B(\lambda,1)}  |f(w)|^2 e^{-\phi(w)}\ dA(z).
	\end{align*}
\end{proof}

\section{Sufficient conditions for interpolation and sampling}\label{sec_suf}
Throughout this section we assume that $(\Lambda, m_\Lambda)$ is a separated set with multiplicities that is compatible with $\phi$.

\subsection{Interpolation}\label{sec:suf_is}

We start by showing that interpolation can be solved locally.

\begin{lemma}\label{lemma:local_interpolant}
	Let $\varepsilon>0$, $\lambda \in \Lambda$ and  $\{c_l:\, l=0, \ldots, m_\Lambda(\lambda)-1 \} \subseteq \bC$. Then there exists $f_\lambda: B(\lambda,\varepsilon) \longrightarrow \bC$ analytic such that
	\begin{align*}
	\dbarstar{f_\lambda}{j}(\lambda) &= c_j,
\qquad 0 \leq j \leq m_\Lambda(\lambda)-1,
	\\
	\intertext{and}
	\left| f_\lambda(z)\right|^2 e^{-\phi(z)} &\leq  C_{\varepsilon} \sum_{l=0}^{m_\Lambda(\lambda)-1} |c_l|^2 e^{-\phi(\lambda)},
	\end{align*}
for all $z\in B(\lambda,\varepsilon)$.
\end{lemma}
\begin{proof}
	We consider $G_\lambda$ as in the Lemma \ref{lemma:bounds_potentials} on the set $B(\lambda,\varepsilon)$. We define
	\begin{equation}\label{eq:local_interpolant_1}f_{\lambda}(z) := p_\lambda(z) e^{G_{\lambda}(z)},\end{equation}
	where \begin{equation}\label{eq:local_interpolant_2}p_\lambda(z) := \sum_{j=0}^{m_{\Lambda}(\lambda)-1} \frac{k_j}{j!} (z-\lambda)^j,\end{equation} and $k_j$ are real coefficients to be chosen so that
	\begin{equation}\label{eq:local_interpolant_3}
	\dbarstar{f}{j}_{\lambda}\left(\lambda\right) = c_{j},
	\end{equation} for all $j\in\{0,\dots,m_\Lambda(\lambda)-1\}$.

	To obtain an explicit formula of $k_j$ we proceed as follows,
	\begin{align}
	\dbarstar{f_\lambda}{j}(w) &= (-1)^j e^{\phi(w)}\partial^j\left( f_\lambda e^{-\phi} \right)(w)  \nonumber
	= (-1)^j e^{\phi(w)}\partial^j\left( p_\lambda e^{G_{\lambda}-\phi} \right)(w) , \nonumber \\
	&= (-1)^j e^{\phi(w)}{\left[\sum_{l=0}^{j} \binom{j}{l} \partial^{\thinspace l} p_\lambda (w)
		\partial^{\thinspace j-l}\left(e^{G_{\lambda}-\phi}  \right)(w)            \right] } .
	\label{eq:k_j}
	\end{align}
The $n$-th derivative of two composite functions can be computed by means of \emph{Fa\`{a} di Bruno's} formula and the \emph{Bell polynomials} (see Equation \eqref{eq:exponential_derivative} and Section \ref{sec:Bell} for definitions and suitable references). More precisely,
	\begin{equation}\label{eq:derivative_egl}
	\partial^{k} e^{G_{\lambda}-\phi}(w)=e^{G_{\lambda}(w)-\phi(w)} B_k\bigg( \partial(G_{\lambda}-\phi)(w),\ \ldots \ ,\
	\partial^k (G_{\lambda}-\phi)(w)         \bigg),
	\end{equation}
	where $B_k$ is the $k-$th complete Bell polynomial. We write $B_{k}^{g}(w):=B_{k}\left(\partial g(w), \ldots, \partial^{k} g(w)\right)$ to shorten notation.

	Since $G_\lambda(\lambda)=0$, substituting \eqref{eq:derivative_egl} into \eqref{eq:k_j}, and evaluating at $w=\lambda$, it follows that
	\begin{equation*}
	\dbarstar{f_\lambda}{j}(\lambda) = (-1)^j\left( k_j + \sum_{l=0}^{j-1} \binom{j}{l}\ k_{l}\
		B_{j-l}^{G_{\lambda}-\phi}(\lambda)\right) .
	\end{equation*}
	Therefore, \eqref{eq:local_interpolant_3} is equivalent to:
	\begin{equation} \label{eq:interpolation_equivalent_condition} k_j = (-1)^j  c_{j} -  {\sum_{l=0}^{j-1} \binom{j}{l}\ k_{l}
	}\ B_{j-l}^{G_{\lambda}-\phi}(\lambda),\end{equation}
for all $j\in\{0,\dots,m_\Lambda(\lambda)-1\}$.
	Solving the recursion yields:
	\begin{equation}\label{eq:explicit_formula_interpolation_coeficients} k_j = {\sum_{l=0}^{j} (-1)^l \binom{j}{l}\   c_{l}
	}\ B_{j-l}^{\phi - G_{\lambda}}(\lambda),\end{equation}
for all $j\in\{0,\dots,m_\Lambda(\lambda)-1\}$; see Section \ref{sec:equation_coef} for a proof.

	\begin{remark}\label{rem:poly_bounded} By  Lemma \ref{lemma:bounds_potentials}, $\left|\partial^j(\phi-G_{\lambda})(\lambda) \right| \leq C_\varepsilon$ for each $j\in\{1,\ldots,m_\Lambda(\lambda)-1\}$. Since $B_{j}^{\phi-G_{\lambda}}(\lambda)$ is a polynomial evaluated at $\Big(\partial(\phi-G_{\lambda})(\lambda),\, \ldots,\, \partial^j(\phi-G_{\lambda})(\lambda)\Big)$, we have $\left|B_{j}^{\phi-G_{\lambda}}(\lambda)\right|\lesssim C_\varepsilon$.
	\end{remark}

	On account of the above remark we conclude that $$|k_j|^2 \lesssim C_\varepsilon {\sum_{l=0}^{j} \   |c_{l}|^2
	}.$$
By Lemma \ref{lemma:bounds_potentials}, we finally obtain
	\begin{align*}
	\left|f_{\lambda}(z)\right|^{2} e^{-\phi(z)}=\left|p_{\lambda}(z)\right|^{2} e^{2\operatorname{Re}{G_\lambda(z)}-\phi(z)} &\lesssim C_\varepsilon \left( {\sum_{l=0}^{m_{\Lambda}(\lambda)-1}
		\left|c_{l}\right|^{2}}\right) e^{-\phi(\lambda)}.
	\end{align*}
\end{proof}

For technical reasons we now study the interpolation problem on $\fock{\phi}$ with respect to
the operator $\bar\partial_{\tilde\phi}^{*}$ associated with a second weight $\tilde\phi$. Provided that both weights are sufficiently smooth, powers of both operators are formally related by
\begin{equation}\label{eq:relation_operators_phi_phi_average}\dbarstarr{f}{j}{\tilde\phi}(z)=\sum_{l=0}^{j} \binom{j}{l}\ \dbarstarr{f}{l}{\phi}(z)\ (-1)^{j-l} \ \partial^{j-l}\left(e^{\phi-\tilde\phi}\right)(z)\ e^{\tilde\phi(z)-\phi(z)},
\qquad f \in \fock{\phi}.
\end{equation}

\begin{prop}\label{prop:interpolation_sufficiency}  Let $\left(\Lambda, m_{\Lambda}\right)$ be a separated set with multiplicities that is compatible with $\phi$. Assume additionally that $\Delta\phi$ is continuous.
	Let $\chi_{r}=\frac{1}{\pi r^{2}} 1_{B(0, r)}$ and $\nu :=\sum_{\lambda \in \Lambda} m_\Lambda (\lambda) \cdot  \delta_{\lambda}$. Suppose that there exists $\delta>0$ and $r>0$ such that
	\begin{align}\label{eq_bbb}
	\pi \nu \star \chi_{r}(z)< \Delta \phi(z)-\delta,
	\end{align}
	for all $z\in\mathbb{C}$. Let $\tilde\phi: \mathbb{C} \to \mathbb{R}$ satisfy
	\begin{align}\label{eq_aaa}
	&\norm{\partial^{j} (\phi - \tilde\phi)}_\infty < \infty, \qquad 0 \leq j \leq n_\Lambda
	\\
	&\label{eq_aaa_bis}
	0 < \inf_{z \in \mathbb{C}} \Delta \tilde \phi(z) \leq
	\sup_{z \in \mathbb{C}} \Delta \tilde \phi(z) < \infty.
	\end{align}

	Then $\left(\Lambda, m_{\Lambda}\right)$ solves the following interpolation problem with respect to $\tilde\phi$: given
	$c\in\ell_\phi^2{\pairsetmult{\Lambda}}$ there exists a function $f \in \fock{\phi} $  such that $
\,\bar\partial_{\tilde\phi}^{*}\,	^{(j)} f(\lambda)=c_{(\lambda,j)}$, for all $\lambda\in\Lambda$ and  $j\in\{0,\dots,{m_{\Lambda}(\lambda)-1}\}$.

(Note that, by \eqref{eq_aaa} and \eqref{eq_aaa_bis}, $\ell_\phi^2{\pairsetmult{\Lambda}}=\ell_{\tilde\phi}^2{\pairsetmult{\Lambda}}$ and $\fock{\phi}=\fock{\tilde\phi}$,
while $\tilde\phi: \mathbb{C} \to \mathbb{R}$ is also compatible with $\left(\Lambda, m_{\Lambda}\right)$.
)
	\end{prop}
\begin{proof}
	Let $c \in\ell_\phi^2\pairsetmult{\Lambda}$.
As in \cite{berndtsson-ortega-95} we construct a weight that has singularities on $\Lambda$. Let $E(z)=\frac{1}{\pi} \log |z|^{2}$ and $$
v=E \star\left(\nu-\nu \star \chi_{r}\right).
$$

Let $\psi :=\phi+\pi v .$ The modified weight $\psi$ has the properties
\begin{align}
\Delta \psi &\geq \pi \nu+\delta \geq \delta, \nonumber \\
\psi &\leq \phi,  \nonumber
\end{align}
and for each $\lambda \in \Lambda$ the following inequality holds
\begin{equation}\label{eq:psi_phi} |\psi(z)- m_\Lambda(\lambda) \log | z-\lambda|^{2}-\phi(z) | \leq C_{\rho},\end{equation}
when $z\in B(\lambda, \rho / 2)$, for some constant $C_\rho$  that depends on $\rho(\Lambda)$ and $r$, although this second dependency is not stressed in the notation. (Notice the factor $m_\Lambda(\lambda)$ in front of the logarithm.)

The first step is to construct the non-analytic interpolant. Fix $\lambda \in\Lambda$.
Note first that, by \eqref{eq_aaa}
and \eqref{eq_aaa_bis}, $(\Lambda, m_\Lambda)$ is compatible with $\tilde\phi$. By Lemma  \ref{lemma:local_interpolant} applied to the weight $\tilde\phi$, there is an analytic function $f_\lambda:B(\lambda,\rho/2) \longrightarrow \bC$ such that for each $j\in\{0,\ldots,m_\Lambda(\lambda)-1\}$,
\begin{align*}
\,\bar\partial_{\tilde\phi}^{*}\,	^{(j)} f_\lambda(\lambda) &= c_{(\lambda,j)},
\end{align*}
and, due to \eqref{eq_aaa}, for all $z \in B(\lambda,\rho/2)$
\begin{align}
\left| f_\lambda(z)\right|^2 e^{-\phi(z)} &\leq  C_\rho \sum_{j=0}^{m_{\Lambda}(\lambda)-1} |c_{(\lambda,j)}|^2 e^{-\phi(\lambda)}, \label{eq:bessel_bound}
\end{align}
where the constant $C_\rho$ is independent of $\lambda$.

Now we patch these functions together. We let $g \in C_{0}^{\infty}$ be 1 on $B(0, \rho / 4)$ and 0 outside $B(0, \rho / 2)$, satisfying $|\overline{\partial} g|<C_{\rho}^{'}$. Then
$$
f(z)=\sum_{\lambda \in \Lambda} f_{\lambda}(z) g(z-\lambda)
$$solves the interpolation problem, although this function might not be analytic.

We consider the problem $\overline{\partial} u= \overline{\partial}f$. By H\"ormander's estimate, there exists a solution satisfying
\begin{align*}
\int_{\mathbb{C}}|u(z)|^{2} e^{-\phi(z)}  \dA(z) &\leq \int_{\mathbb{C}}|u(z)|^{2} e^{-\psi(z)}  \dA(z)
 \lesssim \int_{\mathbb{C}}|\overline{\partial} f(z)|^{2} e^{-\psi(z)}  \dA(z) \\
&\lesssim \int_{\mathbb{C}}|\overline{\partial} f(z)|^{2} e^{-\phi(z)}  \dA(z)
\lesssim \sum_{\lambda\in\Lambda}{\sum_{l=0}^{m_{\Lambda}(\lambda)-1}
	\left|c_{(\lambda,l)}\right|^{2}} e^{-\phi(\lambda)}<\infty.
\end{align*}
For the first inequality we used that $\psi \leq \phi$. For the second we used that $f$ is the solution provided by H\"ormander's estimate\footnote{In principle, $\Delta\psi$ is a distribution; H\"ormander's $L^2$-estimate must be combined with a regularization argument to yield the conclusion.} and that $\Delta \psi>\delta.$  For the third, we used \eqref{eq:psi_phi} and the fact that $\overline{\partial} f$ is zero on a $\rho / 4$ -neighborhood of $\Lambda$. The last inequality follows from the separation of $\Lambda$ and \eqref{eq:bessel_bound}.

The function $u$ is analytic on $B(\lambda, \rho / 4)$, because $\overline{\partial} u= \overline{\partial}f = 0$ on $B(\lambda, \rho / 4)$. In addition, we estimate
\begin{align*}
&\int_{B(\lambda,\rho/4)}\left|\frac{u(z)}{(z-\lambda)^{m_\Lambda(\lambda)}}\right|^{2} \dA(z)  = \int_{B(\lambda,\rho/4)}|u(z)|^{2} e^{-m_\Lambda(\lambda) \log |z-\lambda|^{2}} \dA(z) \\
&\qquad\leq C_\rho \sup_{w\in B(\lambda,\rho/4)} e^{\phi(w)} \int_{B(\lambda,\rho/4)}|u(z)|^{2} e^{-\psi(z)} \dA(z)
\\
&\qquad\leq C_\rho \sup_{w\in B(\lambda,\rho/4)} e^{\phi(w)} \int_{\mathbb{C}}|u(z)|^{2} e^{-\psi(z)} \dA(z) <\infty,
\end{align*}
and conclude that $\partial^{j} u(\lambda) = 0$, for each $\lambda\in\Lambda$ and $0\leq j\leq m_\Lambda(\lambda)-1$.
Moreover, \begin{equation*}
\,\bar\partial_{\tilde\phi}^{*}\,	^{(j)}u(\lambda)
= (-1)^j e^{\tilde\phi(\lambda)}\partial^j\left( u e^{-\tilde\phi} \right)(\lambda) = (-1)^j e^{\tilde\phi(\lambda)}{\left[\sum_{l=0}^{j} \binom{j}{l} \partial^{\thinspace l} u (\lambda)
	\partial^{\thinspace j-l}\left(e^{-\tilde\phi}  \right)(\lambda)   \right] }=0.
\end{equation*}
Thus, the function $f-u$ is holomorphic, belongs to $\fock{\phi}$ and solves the desired interpolation problem
with respect to $\bar\partial^*_{\tilde\phi}$.
\end{proof}

\subsection{Sampling}\label{sec:suf_ss}

\begin{prop}\label{prop:sampling-sufficiency}
Let $\left(\Lambda, m_{\Lambda}\right)$ be a separated set with multiplicities that is compatible with $\phi$. Assume additionally that $\Delta\phi$ is continuous.
Let $\chi_{r}=\frac{1}{\pi r^{2}} 1_{B(0, r)}$ and $\nu :=\sum_{\lambda \in \Lambda} m_\Lambda (\lambda) \cdot  \delta_{\lambda}$. Suppose that there exists $\delta>0$ and $r>0$ such that
$$
\pi \nu \star \chi_{r}(z) > \Delta \phi(z)+\delta,
$$
	for all $z\in\mathbb{C}$. Then $(\Lambda, m_\Lambda)$ is a sampling set for $\fock{\phi}$.
\end{prop}
\begin{proof}
	Let $0 <\varepsilon < \min\left(1,\rho/2\right)$ and $0<t<1$. We follow again \cite{berndtsson-ortega-95}. We construct the weight exactly as they do, except that we add each point mass $m_\Lambda(\lambda)$ times at $\lambda$.   We write
	$$
	\tilde{\nu}(z) := t \sum_{\lambda\in\Lambda} \frac{m_\Lambda(\lambda)}{\pi \varepsilon^{2}} 1_{B(0, \varepsilon)}(z-\lambda).
	$$
	Since $\Delta \phi$ is bounded, we can choose $t$ so close to 1 that
	$$
	\pi \tilde{\nu} \star \chi_{r}>\Delta \phi+\delta / 2 .
	$$
	We let $E(z)=\frac{1}{\pi} \log |z|^{2}$, $v=E \star\left(\tilde{\nu}-\chi_{r} \star \tilde{\nu}\right)$ and $\psi=\phi+\pi v .$
	Note that $$E \star \tilde{\nu}(z)=\sum_{\lambda\in\Lambda}\ \frac{t \ m_\Lambda(\lambda)}{\left(\pi\varepsilon\right)^2} \int_{B(\lambda,\varepsilon)}\log{|z-w|^2} dA(w),$$
	and for $E \star \tilde{\nu} \star \chi_{r}$ we obtain a similar expression. Using that $\log|z-w|^2$ is harmonic on $w\in B(\lambda,\varepsilon)$ when $d(z,\lambda)>r+\varepsilon$ and that the set is relatively separated, we obtain that $|v| \leq C_\varepsilon$. Moreover,
	$$
	\phi-C_{\varepsilon} \leq \psi \leq \phi
	$$ and
	$$
	\left|\psi-m_\Lambda(\lambda)\ t \log \varepsilon^{2}-\phi\right| \leq C_r
	$$
	on $B(\lambda, \varepsilon)$. Notice the factor $m_\Lambda(\lambda)$ in front of the logarithm.

	Let $h\in\fock{\phi}$. As shown in  \cite[Eq. 3]{berndtsson-ortega-95}, we have
	$$
	\delta/2 \int_{\bC}|h(z)|^{2} e^{-\psi(z)} \dA(z) \leq \int_{\bC}|h(z)|^{2} e^{-\psi(z)} d \tilde{\nu}(z) .
	$$
Thus,
	$$
	\int_{\bC}|h(z)|^{2} e^{-\phi(z)} \dA(z) \lesssim C_r \sum_{\lambda \in \Lambda} m_\Lambda(\lambda) \frac{\varepsilon^{-2  m_\Lambda(\lambda) t}}{\varepsilon^{2}} \int_{B(\lambda, \varepsilon)}|h(z)|^{2} e^{-\phi(z)}\dA(z).
	$$
	Let $g_{\lambda}=h e^{-G_{\lambda}},$ where $G_{\lambda}$ is as in Lemma \ref{lemma:bounds_potentials}, with respect to the set $B(\lambda,1)$.
	  Then
	$$
	\int_{B(\lambda, \varepsilon)}|h(z)|^{2} e^{-\phi(z)}\dA(z) \lesssim \int_{B(\lambda, \varepsilon)}\left|g_{\lambda}(z)\right|^{2} e^{-\phi(\lambda)} \dA(z).
	$$
	We now use the $(m_\Lambda(\lambda)-1)$-th order Taylor expansion of $g_\lambda$ on $B(\lambda, \varepsilon)$
	\begin{equation}\label{eq:taylor}
	\left|g_{\lambda}(z)\right|^{2} \lesssim  \sum_{j=0}^{m_\Lambda(\lambda)-1} \varepsilon^{2j}\left|\partial^{j} g_{\lambda}(\lambda)\right|^{2} + \varepsilon^{2\thinspace m_\Lambda(\lambda)} \sup _{z \in B(\lambda, \varepsilon)}\left|\partial^{m_\Lambda(\lambda)} g_{\lambda}(z)\right|^{2} .
\end{equation}
We estimate the $m_\Lambda(\lambda)$-th term by Cauchy estimate
$$
\sup _{z \in B(\lambda, \varepsilon)}\left|\partial^{m_\Lambda(\lambda)} g_{\lambda}(z)\right|^{2} e^{-\phi(\lambda)} \lesssim \int_{B(\lambda, 1)}|h(z)|^{2} e^{-\phi(z)} \dA(z).
$$ 
The remaining terms in \eqref{eq:taylor} are estimated observing that for any $w \in B(\lambda, \varepsilon)$
\begin{equation*}
\partial^{j} g_{\lambda}(w) = \partial^{j} (h e^{-\phi} e^{\phi-G_\lambda}) (w) = e^{\phi(w)-G_\lambda(w)}  \sum_{k=0}^{j} \binom{j}{k} \partial^{k}( h e^{-\phi})(w) \cdot  B_{j-k}^{\phi-G_\lambda}(w),
\end{equation*}
where $B_j$ is the $j-$th Bell polynomial. Evaluating at $\lambda$ we obtain:
$$\partial^{j} g_{\lambda}(\lambda) = \sum_{k=0}^{j} (-1)^k \binom{j}{k} \dbarstar{h}{k}(\lambda) \cdot  B_{j-k}^{\phi-G_\lambda}(\lambda).$$

Thus, for each $j \leq m_\Lambda(\lambda)-1$, and applying Remark \ref{rem:poly_bounded} we have
\begin{align*}
\left|\partial^{j} g_{\lambda}(\lambda)\right|^{2} e^{-\phi(\lambda)} &= \left| \sum_{k=0}^{j} (-1)^k \binom{j}{k} \dbarstar{h}{k}(\lambda) \cdot B_{j-k}^{\phi-G_\lambda}(\lambda)  \right|^{2} e^{-\phi(\lambda)} \lesssim  \sum_{k=0}^{m_\Lambda(\lambda)-1}  |\dbarstar{h}{k}(\lambda)|^2 e^{-\phi(\lambda)}.
\\
\end{align*}

Putting everything together, we obtain
\begin{align}\label{eq:almost_sampling_inequality}
\begin{aligned}
&\int_{\bC}|h|^{2} e^{-\phi} \dA 
\\
&\qquad
\lesssim  C_r \Big( \sum_{\lambda\in\Lambda} \varepsilon^{-2tm_\Lambda(\lambda)} \sum_{j=0}^{m_\Lambda(\lambda)-1}  |\dbarstar{h}{j}(\lambda)|^2     e^{-\phi(\lambda)} +  \sum_{\lambda\in\Lambda} \varepsilon^{2 m_{\Lambda}(\lambda) (1-t)} \int_{B(\lambda,1)}|h|^{2} e^{-\phi} \dA \Big).
\end{aligned}
\end{align}

Since $1 \leq m_{\Lambda}(\lambda)$,
$$\varepsilon^{2 m_{\Lambda}(\lambda) (1-t)} \leq \varepsilon^{2 (1-t)}.$$ The relative separateness of
$\Lambda$ implies
\begin{align*}
\sum_{\lambda\in\Lambda} \varepsilon^{2 m_{\Lambda}(\lambda) (1-t)} \int_{B(\lambda,1)}|h|^{2} e^{-\phi} \dA &\leq \varepsilon^{2  (1-t)} \sum_{\lambda\in\Lambda} \  \int_{B(\lambda,1)} |h|^{2} e^{-\phi} \dA
\\
&\leq \varepsilon^{2  (1-t)} \operatorname{rel}(\Lambda)\, \int_{\bC} |h|^{2} e^{-\phi} \dA.
\end{align*}
Therefore,
\begin{equation*}
\int_{\bC} |h|^{2} e^{-\phi}  \dA \lesssim C_{\varepsilon, r} \sum_{\lambda\in\Lambda} \sum_{j=0}^{m_\Lambda(\lambda)-1}  |\dbarstar{h}{j}(\lambda)|^2 e^{-\phi(\lambda)} + \varepsilon^{2  (1-t)} \operatorname{rel}(\Lambda)\ C_r\ \int_{\bC} |h|^{2} e^{-\phi} \dA.
\end{equation*}
Choosing $\varepsilon$ small enough, the last term can be absorbed into the left-hand side, yielding the desired sampling estimate.
\end{proof}

\section{Necessary conditions for interpolation and sampling}\label{sec_nec}
In this section we reduce the problem of deriving necessary conditions for interpolation and sampling with derivatives to the corresponding problem without derivatives.
We assume throughout this section that $\left(\Lambda, m_{\Lambda}\right)$ is a separated set with multiplicities that is compatible with
the weight $\phi$.

Our arguments are based on inspecting Taylor expansions. The following observation will be used repeatedly.

\begin{remark}\label{rem:equation}
	Let $\varepsilon\in(0,1/4)$, $\lambda \in\bC$ and $F:U\subseteq\bC \longrightarrow \bC$ be holomorphic, where $B(\lambda,1)\subseteq U$. The Taylor expansion of $F$ of degree $n$ at $\lambda$,
	$$F(z)=\sum_{k= 0}^{n} \frac{F^{(k)}(\lambda)}{k!}(z-\lambda)^k+E_n(z),$$
	satisfies
	\begin{equation}\label{eq:F_lamda_prime_2}
	|E_{n}(z)| \lesssim \varepsilon^{n+1}\int_{B(\lambda,1)} |F|\, \dA, \qquad |z-\lambda| \leq \varepsilon.
	\end{equation}
In particular, if $\lambda'$ is such that $|\lambda-\lambda'|=\varepsilon$ and $w \in\bC$ is such that
	\begin{equation*}F^{(n)}(\lambda)=\frac{n!}{(\lambda'-\lambda)^n}\left( w -\sum_{k= 0}^{n-1} \frac{F^{(k)}(\lambda)}{k!}(\lambda'-\lambda)^k \right),\end{equation*}
	 then
	\begin{align}\label{eq:F_lamda_prime}
	|F(\lambda')-w|^2  \lesssim \varepsilon^{2(n+1)}\int_{B(\lambda,1)} |F|^2 dA.
	\end{align}
\end{remark}

As a first step towards the necessary conditions we compute derivatives of weighted functions.

\begin{lemma}\label{lemma:relationship} Let $f\in \fock{\phi}$, $\lambda\in\Lambda$, and $\varepsilon>0$. We write $\phi=h_{\lambda}+G[\Delta\phi]$ on $B(\lambda,\varepsilon)$ as in Lemma \ref{lemma:bounds_potentials}. $H_\lambda:B(\lambda,\varepsilon) \longrightarrow \bC$ is a holomorphic function such that $2\operatorname{Re}{H_\lambda}=h_\lambda$.
	 Then
\begin{equation}\label{eq:d_dbar} \partial^{k} \left(fe^{-H_\lambda}\right) =  \sum_{j=0}^{k} (-1)^j \binom{k}{j} \, \dbarstar{f}{j}\,
	B_{k-j}^{G[\Delta\phi]}
	e^{-H_\lambda} .\end{equation}
	\end{lemma}
	\begin{proof}
Since $\partial(e^{-\overline{H_\lambda}})=\partial(\overline{e^{-H_\lambda}})=\overline{\left(\overline{\partial} {e^{-H_\lambda}}    \right)}\equiv 0,$
		\begin{align*}
\partial^{k}\left(fe^{-H_\lambda}\right) &= \partial^{k} \left(fe^{-\phi}e^{G[\Delta\phi]}e^{\overline{H_\lambda}} \right) =  \partial^{k} \left(fe^{-\phi}e^{G[\Delta\phi]} \right) e^{\overline{H_\lambda}} \\
		&= \sum_{j=0}^{k} \binom{k}{j} \partial^{j} \left(fe^{-\phi}\right) \partial^{k-j}\left( e^{G[\Delta\phi]} \right) e^{\overline{H_\lambda}}
		\\&= \sum_{j=0}^{k} \binom{k}{j} \partial^{j} \left(fe^{-\phi}\right) e^{\phi} \partial^{k-j}\left( e^{G[\Delta\phi]} \right) e^{-H_\lambda-G[\Delta\phi]} \\
		&= \sum_{j=0}^{k} (-1)^j\,  \binom{k}{j} \, \dbarstar{f}{j}\, \partial^{k-j} \left( e^{G[\Delta\phi]} \right)\, e^{-H_\lambda-G[\Delta\phi]}
		\\&= \sum_{j=0}^{k} (-1)^j\,  \binom{k}{j} \, \dbarstar{f}{j}\,
e^{G[\Delta\phi]} B_{k-j}^{G[\Delta\phi]}
 e^{-H_\lambda-G[\Delta\phi]} \\
 &= \sum_{j=0}^{k} (-1)^j\,  \binom{k}{j} \, \dbarstar{f}{j}\,
B_{k-j}^{G[\Delta\phi]}
 e^{-H_\lambda} .
		\end{align*}
	\end{proof}

\subsection{Interpolation}\label{sec:nec_is}

We now show that a weighted interpolating set can be modified to reduce its maximal multiplicity while preserving its density and interpolating property.

\begin{prop}\label{prop:is_decreasing}
Let $\left(\Lambda, m_{\Lambda}\right)$ be a separated set with multiplicities that is compatible with $\phi$. Suppose that
$(\Lambda,m_{\Lambda})$ is interpolating for $\fock{\phi}(\bC)$, and that $\sup_{\lambda\in\Lambda} m_{\Lambda}(\lambda) = {\nLambda}+1 \geq 2$. Then there exists a separated and interpolating set $(\tilde\Lambda,m_{\tilde\Lambda})$ such that $\sup_{\lambda\in\tilde{\Lambda}} m_{\tilde{\Lambda}}={\nLambda}$ and $D^{\pm}(\Lambda, m_{\Lambda})=D^{\pm}(\tilde\Lambda, m_{\tilde\Lambda})$.
\end{prop}
\begin{proof}
\noindent {\bf Step 1}. \emph{(Definition of the new set).}

	Let $0<\varepsilon<\min\{\rho(\Lambda)/2,1/4\}$. We define $$\Lambda_{\text{max}} := \left\{\lambda\in\Lambda \, : \, m_{\Lambda}(\lambda)=\sup_{z\in\Lambda} m_{\Lambda}(z)={\nLambda}+1 \right\}.$$ For each $\lambda \in \Lambda_{\text{max}}$ we choose  $\lambda'\in\bC$ such that $|\lambda-\lambda'|=\varepsilon$. We define $\Lambda'=\{\lambda' : \lambda\in\Lambda_{\text{max}} \}$. Since
$	\varepsilon< \rho(\Lambda)/2$,
it is clear that the map $\lambda\mapsto\lambda'$ is injective and that $\lambda'\not\in\Lambda$.

Now we consider the set $\tilde{\Lambda}=\Lambda\cup\Lambda'$ and the function $m_{\tilde\Lambda}:\tilde{\Lambda}\longrightarrow\bN$ defined by
\[
m_{\tilde\Lambda}(z) =
\begin{cases}
m_{\Lambda}(z) &\quad\text{if }z\in\Lambda\text{ and } m_{\Lambda}(z)\leq {\nLambda},\\
{\nLambda} &\quad\text{if }z\in\Lambda\text{ and } m_{\Lambda}(z)={\nLambda}+1,\\
1 &\quad\text{if }z\in\Lambda'.\\
\end{cases}
\]

Since $\varepsilon < \rho(\Lambda)/2$ it follows easily that $\tilde\Lambda$ is separated. It is also clear that $\sup m_{\tilde\Lambda}={\nLambda}$.

\noindent {\bf Step 2}. We show that if $\tilde{a} \in \ell_\phi^2\pairsetmult{\tilde\Lambda}$, then there exists $f\in\fock{\phi}$ satisfying:
		\begin{enumerate}[label=\textbf{(Q.\arabic*)}]
			\item \label{cond:interpolation_1}$\norm{f}_{\fock{\phi}}\leq C_\Lambda  \varepsilon^{-{\nLambda}}\norm{\tilde{a}}_{\ell_\phi^2\pairsetmult{\tilde\Lambda}}$,
			\item \label{cond:interpolation_2} $\dbarstar{f}{j}(\lambda) = \tilde{a}_{(\lambda,j)}$, for each $\lambda \in \Lambda$ and $0\leq j \leq{\min\{{\nLambda}-1, m_\Lambda(\lambda)-1  \}}$, 
			\item \label{cond:interpolation_3}
			$\norm{f - \tilde{a}_{(\cdot , 0)}}_{\ell_{\phi}^2(\Lambda')}\leq C_\Lambda \varepsilon^{{\nLambda}+1} \norm{f}_{\fock{\phi}}.$\footnote{To unload the notation, we write $\norm{f}_{\ell^2_\phi(\Lambda)}$ instead of
				$\norm{f|\Lambda}_{\ell^2_\phi(\Lambda)}$}
		\end{enumerate}

To prove the claim, for each $\lambda\in\Lambda$ we write $\phi=h_{\lambda}+G[\Delta\phi]$ on $B(\lambda,1)$ as in Theorem \ref{thm:Riesz_decomp}.
Since $h_\lambda$ is harmonic, there exists $H_\lambda:B(\lambda,1) \longrightarrow \bC$ analytic such that $2\operatorname{Re}{H_\lambda}=h_\lambda$.

For each $\lambda\in\Lambda_{\text{max}}$ we let $b_\lambda$ be defined by
		\begin{align}\label{eq:b_lambda_equation}
		\begin{aligned}
		(-1)^{\nLambda} \ b_{\lambda}\  e^{-H_\lambda(\lambda)} &= \frac{ {\nLambda}!}{(\lambda'-\lambda)^{\nLambda}} \Big(\tilde{a}_{(\lambda',0)} e^{-H_\lambda(\lambda')}
		\\
		&\qquad-\sum_{k=0}^{{\nLambda}-1} \sum_{j=0}^{k}
		(-1)^j \frac{\binom{k}{j}}{k!} \tilde{a}_{(\lambda,j)} B_{k-j}^{G[\Delta\phi]}(\lambda)   e^{-H_\lambda(\lambda)} (\lambda'-\lambda)^k  \Big)\\
		 &\qquad- \sum_{j=0}^{{\nLambda}-1} (-1)^j  \binom{{\nLambda}}{j}  \tilde{a}_{(\lambda,j)} B_{{\nLambda}-j}^{G[\Delta\phi]}(\lambda) e^{-H_\lambda(\lambda)}.
		\end{aligned}
		\end{align}
		The sequence satisfies
\begin{equation}\label{eq_mmm}
		\left|b_{\lambda}\right|^2  e^{-\phi(\lambda)} \lesssim \left|b_{\lambda}  e^{-H_\lambda(\lambda)}\right|^2 \lesssim \varepsilon^{-2{\nLambda}}\left(|\tilde{a}_{(\lambda',0)}|^2 e^{-\phi(\lambda')}+\sum_{j=0}^{{\nLambda}-1} |\tilde{a}_{(\lambda,j)}|^2 e^{-\phi(\lambda)}\right).\end{equation}
		We define a new sequence $a=\left\{ a_{(\lambda,j)} \right\}_{{\lambda\in\Lambda,\, 0\leq j \leq m_{\Lambda}(\lambda)-1 }} \subseteq \bC$ by:

			\[
a_{(\lambda,j)} :=
		\begin{cases}
				b_\lambda &\quad\text{if }\lambda\in\Lambda_{\text{max}}\text{ and } j={\nLambda},
\\
		\tilde{a}_{(\lambda,j)} &\quad\text{otherwise.}
		\end{cases}
		\]
		By \eqref{eq_mmm},
		\begin{equation}\label{eq_mmmm}
		\begin{split}
		\norm{a}_{ \ell_\phi^2\pairsetmult{\Lambda}}^2 \lesssim \varepsilon^{-2\nLambda}   \norm{\tilde{a}}_{\ell_\phi^2\pairsetmult{\tilde\Lambda}}^2       <\infty.
		\end{split}
		\end{equation}
		Since $(\Lambda,m_\Lambda)$ is an interpolating set, there exists a function $f \in\fock{\phi}$ such that \begin{itemize}
			\item $\norm{f}_{\fock{\phi}} \leq C_{\Lambda} \norm{a}_{ \ell_\phi^2\pairsetmult{\Lambda}}$,
			\footnote{As in the unweighted case, and with the same argument, the interpolation problem, if solvable, can be solved with norm control.}
			\item $\dbarstar{f}{j}(\lambda)=\tilde{a}_{(\lambda, j)}$, if $\lambda\in\Lambda$ and $0\leq j \leq{\min\{{\nLambda}-1, m_\Lambda(\lambda)-1  \}}$,
			\item $\dbarstar{f}{{\nLambda}}(\lambda)=b_{\lambda}$, if $\lambda\in\Lambda_{\text{max}}$.
		\end{itemize}
The first two conditions, together with \eqref{eq_mmmm} yield \ref{cond:interpolation_1} and \ref{cond:interpolation_2}.
To check \ref{cond:interpolation_3}, we let $\lambda\in \Lambda_{\text{max}}$, and note that, in terms of $f$,
\eqref{eq:b_lambda_equation} reads:
	\begin{align*}
&\sum_{j=0}^{{\nLambda}} (-1)^j \binom{{\nLambda}}{j}  \dbarstar{f}{j}(\lambda) B_{{\nLambda}-j}^{G[\Delta\phi]}(\lambda) e^{-H_\lambda(\lambda)}
\\
&\qquad=
 \frac{ {\nLambda}!}{(\lambda'  -\lambda)^{\nLambda}} \Big(\tilde{a}_{(\lambda',0)} e^{-H_\lambda(\lambda')}  -\sum_{k=0}^{{\nLambda}-1} \sum_{j=0}^{k}
	(-1)^j \frac{\binom{k}{j}}{k!} \dbarstar{f}{j}(\lambda) B_{k-j}^{G[\Delta\phi]}(\lambda)   e^{-H_\lambda(\lambda)} (\lambda'-\lambda)^k  \Big).
	\end{align*}
Applying \eqref{eq:d_dbar} to the last equation we obtain
	\begin{equation}\label{eq:n_derivative_feh}
\begin{split}
\partial^{{\nLambda}}({fe^{-H_\lambda}})(\lambda)  = \frac{ {\nLambda}!}{(\lambda'  -\lambda)^{\nLambda}} &\Big(\tilde{a}_{(\lambda',0)} e^{-H_\lambda(\lambda')} -\sum_{k=0}^{{\nLambda}-1}
\frac{\partial^{k}({fe^{-H_\lambda}})(\lambda)}{k!}  (\lambda'-\lambda)^k  \Big).
\end{split}
\end{equation}
We apply Remark \ref{rem:equation} to $F=fe^{-H_\lambda}$ and conclude from \eqref{eq:F_lamda_prime} that
 $$\left|f(\lambda')-\tilde{a}_{(\lambda',0)} \right|^2 e^{-\phi(\lambda')} \lesssim \left|f(\lambda')e^{-H_\lambda(\lambda')}-\tilde{a}_{(\lambda',0)} e^{-H_\lambda(\lambda')}\right|^2
  \lesssim \varepsilon^{2({\nLambda}+1)} \norm{f}_{L^2(B(\lambda',1),e^{-\phi})}^2.$$
 Therefore,
 \begin{align*}
 \norm{f - \tilde{a}_{(\cdot , 0)}}_{\ell_\phi^2(\Lambda')}^2 &\lesssim \sum_{\lambda' \in\Lambda'} \varepsilon^{2({\nLambda}+1)} \norm{f}_{L^2(B(\lambda',1),e^{-\phi})}^2
 \\&\lesssim \operatorname{rel}(\Lambda)\, \varepsilon^{2({\nLambda}+1)}\,\norm{f}_{\fock{\phi}}^2,
 \end{align*}
which yields \ref{cond:interpolation_3}.

\noindent {\bf Step 3}. We show that $\left(\tilde\Lambda,m_{\tilde\Lambda}\right)$ is an interpolating set.

Let $\tilde{a} \in \ell_\phi^2{\pairsetmult{\tilde\Lambda}}$. We define inductively $\tilde a^{(k)}\in \ell_\phi^2{\pairsetmult{\tilde\Lambda}}$ and $f_k\in\fock{\phi}$. Let $\tilde{a}^{(1)}:=\tilde{a}$. By Step 2, there exists $f_1\in\fock{\phi}$ satisfying \ref{cond:interpolation_1}, \ref{cond:interpolation_2} and \ref{cond:interpolation_3} for $\tilde{a}^{(1)}$.
Let $k\geq 2$ and suppose that $\tilde a^{(k-1)} \in \ell_\phi^2{\pairsetmult{\tilde\Lambda}}$
and $f_{k-1}\in\fock{\phi}$ are already defined and satisfy
\ref{cond:interpolation_1}, \ref{cond:interpolation_2} and \ref{cond:interpolation_3} with respect to $\tilde a^{(k-1)}$. Define ${\tilde{a}}^{(k)}\in \ell_\phi^2{\pairsetmult{\tilde\Lambda}}$ by
\begin{equation}  \label{eq:recursion}
	{\tilde{a}}_{(\lambda,j)}^{(k)} :=
	\begin{cases}
	0 &\quad\text{if }\lambda\in\Lambda\text{ and } 0\leq j \leq m_{\tilde\Lambda}(\lambda)-1,\\
	{\tilde{a}}_{(\lambda,0)}^{(k-1)}-f_{k-1}(\lambda) &\quad\text{if }\lambda\in\Lambda'\text{ and } j=0.\\
	\end{cases}
\end{equation}
Since $\norm{\tilde{a}^{(k)}}_{\ell_\phi^2\pairsetmult{\tilde\Lambda}}=\norm{f_{k-1}-{\tilde{a}}_{(\cdot,0)}^{(k-1)}}_{\ell_\phi^2(\Lambda')}\leq C_{\Lambda} \varepsilon^{{\nLambda}+1}\norm{f_{k-1}}_{\fock{\phi}}<\infty$,
the sequence is indeed well-defined, and we can apply Step 2 for $\tilde{a}^{(k)}$. Let $f_k$ satisfy \ref{cond:interpolation_1}, \ref{cond:interpolation_2} and \ref{cond:interpolation_3} with respect to $\tilde{a}^{(k)}$.

The constructed sequences satisfy
	$$\norm{f_{k}}_{\fock{\phi}}\leq C_\Lambda  \varepsilon^{-{\nLambda}}\norm{\tilde{a}^{(k)}}_{\ell_\phi^2\pairsetmult{\tilde\Lambda}}\leq\varepsilon\, C_\Lambda^{'} \norm{f_{k-1}}_{\fock{\phi}},$$ and,
	$$\norm{\tilde{a}^{(k)}}_{\ell_\phi^2\pairsetmult{\tilde\Lambda}} \leq\varepsilon\, C_\Lambda^{'}\, \norm{\tilde{a}^{(k-1)}}_{\ell_\phi^2{\pairsetmult{\tilde\Lambda}}}.$$
	Hence, if $\varepsilon\, C_\Lambda^{'} < 1$, $f:=\sum_{k\in\bN}f_k$ converges in $\fock\phi{}$.

	To see that interpolates with the right values 
	on $\Lambda^{'}$,
	we observe that for each $\lambda\in\Lambda_{\text{max}}$,
\begin{align*}
\left|\sum_{k=1}^{N} f_k(\lambda')-\tilde{a}^{(1)}_{(\lambda',0)}\right|^2 e^{-\phi(\lambda')} &= \left|\tilde{a}^{(N+1)}_{(\lambda',0)}\right|^2 e^{-\phi(\lambda')} \\ &\leq \norm{\tilde{a}^{(N+1)}}^2_{\ell_\phi^2(\Lambda')} \xrightarrow[N \to \infty]{} 0.
\end{align*}
Thus, $f(\lambda')=\tilde{a}^{(1)}_{(\lambda',0)}$, when $\lambda \in \Lambda_{\text{max}}$. For $\lambda\in \tilde{\Lambda} \setminus\Lambda'$ we argue as follows.
Since $\sum_{k\in\bN}f_k$ converges in $\fock\phi{}$, by Lemma \ref{lemma:compare_dbarstar_normf}, $\dbarstar{f}{j} = \sum_{k\in\bN} \dbarstar{f_k}{j}$. As $\dbarstar{f_k}{j}(\lambda)=0$ for $k\geq2$, $\dbarstar{f}{j}(\lambda)=\dbarstar{f_1}{j}(\lambda)=  \tilde{a}_{(\lambda,j)}$,
	if $0 \leq j \leq m_{\tilde\Lambda}(\lambda)-1$.
	
	We have thus shown that $(\tilde\Lambda,m_{\tilde\Lambda})$ is an interpolating set.

{\bf Step 4.} We show that $D^{\pm}(\Lambda, m_{\Lambda})=D^{\pm}(\tilde\Lambda, m_{\tilde\Lambda})$.

 The number of points of the set $\Lambda$ in the disk of center $z$ and radius $r$, counted with multiplicities, satisfies $
\numberofpoints{z}{r}{\Lambda}{\Lambda} \leq
\numberofpoints{z}{r+\varepsilon}{\tilde\Lambda}{\tilde\Lambda}$.
Since $0<m \leq \Delta \phi \leq M$ then for any $\delta>0$, there exists $r_{\delta}$ such that if $r > r_{\delta}$, then  $\sup_{z \in \mathbb{C}} \left|\frac{\int_{B(z, r+\varepsilon)} \Delta \phi}{\int_{B(z, r)} \Delta \phi}-1\right|<\delta$. Thus,
\begin{align*}
(1+\delta)^{-1} \cdot
D_{\phi}^{-}\left(\Lambda, m_{\Lambda}\right)&\leq\liminf_{r \rightarrow \infty} \inf _{z \in \mathbb{C}} \frac{\numberofpoints{z}{r+\varepsilon}{\tilde\Lambda}{\tilde\Lambda}}{\int_{B(z, r+\varepsilon)} \Delta \phi\, \dA}=D_{\phi}^{-}\left(\tilde\Lambda, m_{\tilde\Lambda}\right).
\end{align*}
We let $\delta \rightarrow 0$.
Exchanging the roles of $\Lambda$ and $\tilde\Lambda$ we also see that
 $D_{\phi}^{-}\left(\tilde\Lambda, m_{\tilde\Lambda}\right) \leq D_{\phi}^{-}\left(\Lambda, m_{\Lambda}\right) $. Analogously, $D_{\phi}^{+}\left(\Lambda, m_{\Lambda}\right) = D_{\phi}^{+}\left(\tilde\Lambda, m_{\tilde\Lambda}\right)$.
\end{proof}

\subsection{Sampling}\label{sec:nec_ss}

The modification of sampling sets to reduce their multiplicity is based on the following perturbation lemma.

\begin{lemma}\label{lemma:dbarstar_n} Let $\lambda,\lambda' \in \bC$ and $0<\varepsilon<1/4$. If
	$f\in\fock{\phi}$ and  $|\lambda'-\lambda|=\varepsilon$, there exists a constant $C_{\varepsilon}$ such that
	$$|\dbarstar{f}{{\nLambda}}(\lambda)|^2 e^{-\phi(\lambda)} \lesssim C_{\varepsilon}\left( |f(\lambda')|^2 e^{-\phi(\lambda')} + \sum_{j=0}^{{\nLambda}-1}  |\dbarstar{f}{j}(\lambda)|^2 e^{-\phi(\lambda)} \right) +  \varepsilon \norm{f}_{L^2(B(\lambda,1),e^{-\phi})}^2 .$$
\end{lemma}

\begin{proof}
	We apply Theorem \ref{thm:Riesz_decomp} on ${B(\lambda,1)}$ and obtain $\phi(z)=h_\lambda(z)+G[\Delta\phi](z)$. Since $h_\lambda$ is harmonic, there exists $H_\lambda:B(\lambda,1)\longrightarrow\bC$ holomorphic such that $\operatorname{Re}{H_\lambda}=h_\lambda/2$. Therefore,	\begin{equation*}
	\begin{aligned}
	\left|\dbarstar{f}{{\nLambda}}(\lambda)\right|^2 e^{-\phi(\lambda)} &=
	\left|\left(\partial^{\nLambda}(fe^{-\phi})e^{\phi} \right)(\lambda)\right|^2 e^{-\phi(\lambda)}
	\\
	&=   \left|\left(\partial^{\nLambda}\left( f e^{-\phi}\right)\right)(\lambda)\right|^2 e^{\phi(\lambda)}    \\
	&=   \left|\left(\partial^{\nLambda} \left(f e^{-H_\lambda-\overline{H_\lambda}-G[\Delta\phi]}\right)\right)(\lambda)\right|^2 e^{\phi(\lambda)}.
	\end{aligned}\end{equation*}
	Since $\partial(e^{-\overline{H_\lambda}})=\partial(\overline{e^{-H_\lambda}})=\overline{\left(\partial_{\overline{z}} {e^{-H_\lambda}}    \right)}\equiv 0,$
	\begin{equation}\label{eq:lemma_ineq_2}
	\begin{aligned}
	\left|\partial^{\nLambda} \left(f e^{-H_\lambda-\overline{H_\lambda}-G[\Delta\phi]}\right)(\lambda)\right|^2 e^{\phi(\lambda)} &=   \left|\partial^{\nLambda} \left(f e^{-H_\lambda-G[\Delta\phi]}\right)(\lambda)\right|^2 \left|e^{-\overline{H_\lambda(\lambda)}} \right|^2 e^{\phi(\lambda)}  \\
	&=   \left|\partial^{\nLambda} \left(f e^{-H_\lambda-G[\Delta\phi]}\right)(\lambda)\right|^2 e^{-h_\lambda(\lambda)} e^{\phi(\lambda)}
	\\
	&\lesssim
	\left|\partial^{\nLambda} \left(f e^{-H_\lambda-G[\Delta\phi]}\right)(\lambda)\right|^2, 
	\end{aligned}\end{equation}
	where we used that $-h_\lambda(\lambda)+\phi(\lambda)=G[\Delta\phi](\lambda)$ is uniformly bounded, by Lemma \ref{lemma:bounds_potentials}.
	In addition, by Lemma \ref{lemma:bound_partial_f_eh},
	\begin{align}
	\label{eq:sum_derivatives}
	\left|\dbarstar{f}{{\nLambda}}(\lambda)\right|^2 e^{-\phi(\lambda)}\lesssim
	\left|\partial^{\nLambda} (f e^{-H_\lambda-G[\Delta\phi]})(\lambda)\right|^2 
	&\lesssim \sum_{j=0}^{{\nLambda}} \left|\partial^j (f e^{-H_\lambda})(\lambda)\right|^2 .
	\end{align}

	We proceed to estimate $\left|(\partial^{\nLambda} (f e^{-H_\lambda}))(\lambda)\right|$.
 The Taylor expansion for $fe^{-H_\lambda}$ of order ${\nLambda}$ at $\lambda$ evaluated at $\lambda'$ is
	$$\left(fe^{-H_\lambda}\right)(\lambda')=\sum_{k= 0}^{\nLambda} \frac{\partial^k{(fe^{-H_\lambda})}(\lambda)}{k!}(\lambda'-\lambda)^k+E_{\nLambda}(\lambda').$$

	Hence,
	\begin{align*}
	&{\partial^{\nLambda} \left(fe^{-H_\lambda}\right)}(\lambda)
	\\
	&\quad=\frac{{\nLambda}!}{(\lambda'-\lambda)^{\nLambda}}{(fe^{-H_\lambda})}(\lambda') - \frac{{\nLambda}!}{(\lambda'-\lambda)^{\nLambda}} \sum_{k=0}^{{\nLambda}-1} \frac{{\partial^k (fe^{-H_\lambda})}(\lambda)}{k!}(\lambda'-\lambda)^k-\frac{{\nLambda}!}{(\lambda'-\lambda)^{\nLambda}}E_{\nLambda}(\lambda'),
	\end{align*}
and for $|\lambda'-\lambda|=\varepsilon$,
	$$\left|{\partial^{\nLambda}(fe^{-H_\lambda})}(\lambda)\right|^2 \lesssim C_{\varepsilon, {\nLambda}} \left(|{(fe^{-H_\lambda})}(\lambda')|^2 + \sum_{k=0}^{{\nLambda}-1} |\partial^k(fe^{-H_\lambda})(\lambda)|^2\right)+\frac{|E_{\nLambda}(\lambda')|^2}{\varepsilon^{2n}}.$$

	We use the last estimate, together with 
	Lemma \ref{lemma:bound_partial_f_eh} and the fact that $\phi(z)-h_\lambda(z)$ is bounded, and \eqref{eq:sum_derivatives} to obtain:
	\begin{align*}
&\left|\dbarstar{f}{{\nLambda}}(\lambda)\right|^2 e^{-\phi(\lambda)} \lesssim
\sum_{j=0}^{{\nLambda}} \left|\partial^j (f e^{-H_\lambda})(\lambda)\right|^2 
\\
&\quad\lesssim C_{\varepsilon,{\nLambda}} \left|(f e^{-H_\lambda})(\lambda')\right|^2 + {C_{\varepsilon,{\nLambda}}} \sum_{j=0}^{{\nLambda}-1} \left|(\partial^j (f e^{-H_\lambda}))(\lambda)\right|^2 +\frac{|E_{\nLambda}(\lambda')|^2}{\varepsilon^{2n}}
\\
	&\quad\lesssim
 C_{\varepsilon,{\nLambda},\phi}\Big( \left|(f e^{-H_\lambda})(\lambda')\right|^2 + \sum_{j=0}^{{\nLambda}-1} \left|	\partial^j \left(f e^{-H_\lambda-G[\Delta\phi]} \right)(\lambda) \right|^2\Big) +\frac{|E_{\nLambda}(\lambda')|^2}{\varepsilon^{2n}}	\\
	&\quad\lesssim
	C_{\varepsilon,{\nLambda},\phi}^{'} \Big( \left|(f e^{-H_\lambda})(\lambda')\right|^2 + \sum_{j=0}^{{\nLambda}-1} \left|	\partial^j \left(f e^{-H_\lambda-G[\Delta\phi]} \right)(\lambda) \right|^2 e^{-h_\lambda(\lambda)+\phi(\lambda)}    \Big) +\frac{|E_{\nLambda}(\lambda')|^2}{\varepsilon^{2n}}  \\
	&\quad= C_{\varepsilon,{\nLambda},\phi}^{'} \Big( \left|f (\lambda')\right|^2 e^{-h_\lambda(\lambda')} + \sum_{j=0}^{{\nLambda}-1} \left|	\partial^j \left(f e^{-\phi} \right)(\lambda) \right|^2 e^{\phi(\lambda)}    \Big) +\frac{|E_{\nLambda}(\lambda')|^2}{\varepsilon^{2n}} \\
	&\quad\leq C_{\varepsilon,{\nLambda},\phi}^{''} \Big( \left|f (\lambda')\right|^2 	e^{-\phi(\lambda')} + \sum_{j=0}^{{\nLambda}-1} \left|	\partial^j \left(f e^{-\phi} \right)(\lambda) \right|^2 e^{\phi(\lambda)}    \Big) +\frac{|E_{\nLambda}(\lambda')|^2}{\varepsilon^{2n}}\nonumber \\
	&\quad= C_{\varepsilon,{\nLambda},\phi}^{''} \Big( \left|f (\lambda')\right|^2 	e^{-\phi(\lambda')} + \sum_{j=0}^{{\nLambda}-1} \left|	\partial^j \left(f e^{-\phi} \right)(\lambda) e^{\phi(\lambda)} \right|^2 e^{-\phi(\lambda)}    \Big) +\frac{|E_{\nLambda}(\lambda')|^2}{\varepsilon^{2n}} .
	\end{align*}
Moreover, if $|\lambda'-\lambda|=\varepsilon$, the remainder of the Taylor expansion satisfies
\eqref{eq:F_lamda_prime_2} with $F=f e^{-H_\lambda}$, and, therefore,
	\begin{align*}
	\left|\dbarstar{f}{{\nLambda}}(\lambda) \right|^2 e^{-\phi(\lambda)} &\lesssim C_{\varepsilon} \Big(|f(\lambda')|^2 e^{-\phi(\lambda')} + \sum_{j=0}^{{\nLambda}-1}|\dbarstar{f}{j}(\lambda)|^2 e^{-\phi(\lambda)}\Big)
	+ \varepsilon^{2} \norm{f}_{L^2(B(\lambda,1),e^{-\phi})}^2.
	\end{align*}
	The claim follows since $\varepsilon^2 < \varepsilon$.
\end{proof}

We can now modify a sampling set to reduce its maximal multiplicity without altering its sampling property.

\begin{prop}\label{prop:ss_decreasing}
Let $\left(\Lambda, m_{\Lambda}\right)$ be a separated set with multiplicities that is compatible with $\phi$.
Assume that $(\Lambda,m_{\Lambda})$ is sampling for $\fock{\phi}(\bC)$, and that $\sup_{\lambda\in\Lambda} m_{\Lambda}(\lambda) = {\nLambda}+1 \geq 2$. Then there exists a separated set with multiplicities $(\tilde\Lambda, m_{\tilde\Lambda})$ that is sampling for \fock{\phi}, and such that $\sup_{\lambda\in\tilde\Lambda} m_{\tilde\Lambda}(\lambda)={\nLambda}$ and $D^{\pm}(\Lambda, m_{\Lambda})=D^{\pm}(\tilde\Lambda, m_{\tilde\Lambda})$.
\end{prop}

\begin{proof}
Since $(\Lambda,m_{\Lambda})$ is a sampling set for $\fock{\phi}(\bC)$,
\begin{align}\label{eq:sampling_inequality}
	\norm{f}^2_{\fock{\phi}} \lesssim \sum_{\lambda \in \Lambda}  \sum_{j=0}^{m_{\Lambda}(\lambda)-1} |\dbarstar{f}{j}(\lambda)|^2 e^{-\phi(\lambda)}.
	\end{align}

	Let $0<\varepsilon<\min\{\rho(\Lambda)/2,\ 1/4\}$ and consider $\Lambda_{\text{max}}$, $\Lambda^{\prime}$, $\tilde\Lambda$ and $m_{\tilde\Lambda}$ as in the proof of Proposition \ref{prop:is_decreasing}. Hence, $D^{\pm}(\Lambda, m_\Lambda)=D^{\pm}(\tilde\Lambda, m_{\tilde\Lambda})$, and $\tilde\Lambda$ is separated. As a consequence, $(\tilde\Lambda, m_{\tilde\Lambda})$ satisfies an upper sampling bound. We now show that if $\varepsilon$ is small enough, a lower sampling bound also holds for $(\tilde\Lambda, m_{\tilde\Lambda})$.

	For $\lambda\in\Lambda_{\text{max}}$, we have $m_{\tilde\Lambda}(\lambda)=n_\Lambda=m_{\Lambda}(\lambda)-1$ and, by Lemma \ref{lemma:dbarstar_n}, \begin{equation*}\sum_{j=0}^{{m_{\Lambda}(\lambda)-1}} |\dbarstar{f}{j}(\lambda)|^2 e^{-\phi(\lambda)} \lesssim C_{\phi,{\nLambda},\varepsilon}\Big( |f(\lambda')|^2 e^{-\phi(\lambda')} + \sum_{j=0}^{{m_{\tilde\Lambda}(\lambda)}-1}  |\dbarstar{f}{j}(\lambda)|^2 e^{-\phi(\lambda)} \Big) + \varepsilon \norm{f}_{L^2(B(\lambda,1),e^{-\phi})}^2 . \end{equation*}
	On the other hand, for $\lambda\in\Lambda\setminus\Lambda_{\text{max}}$,
	$m_{\Lambda}(\lambda)=m_{\tilde\Lambda}(\lambda)$, and the same equation is trivially true.
Therefore, by \eqref{eq:sampling_inequality}, \begin{equation} \norm{f}^2_{\fock{\phi}} \lesssim \sum_{\lambda \in \Lambda}  \sum_{j=0}^{m_{\Lambda}(\lambda)-1} |\dbarstar{f}{j}(\lambda)|^2 e^{-\phi(\lambda)} \lesssim C_{\varepsilon,\phi,{\nLambda}}^\prime \sum_{\tilde\lambda \in \tilde\Lambda}  \sum_{j=0}^{m_{\tilde\Lambda}(\tilde\lambda)-1} |\dbarstar{f}{j}(\tilde\lambda)|^2 e^{-\phi(\tilde\lambda)} + \varepsilon \operatorname{rel}(\Lambda) \norm{f}^2_{\fock{\phi}}. \end{equation}
	Thus, taking $\varepsilon$ small enough, the term $\varepsilon \operatorname{rel}(\Lambda) \norm{f}^2_{\fock{\phi}}$ can be absorbed into the left-hand side, yielding the desired lower sampling bound.
\end{proof}

\section{Proof of the main results}\label{sec:proof_main_result}

\begin{proof}[Proof of Theorem A]
	Suppose that $(\Lambda, m_\Lambda)$ is set of interpolation for $\fock{\phi}$. Then the unweighted set $\Lambda$ is also a set of interpolation for $\fock{\phi}$, and therefore separated; see, e.g., \cite[Prop. 3]{ortega-seip-98}. Repeated application of Proposition \ref{prop:is_decreasing} shows that there exists a set of interpolation $\tilde\Lambda\subseteq\bC$ without multiplicity such that $D_\phi^{+}(\Lambda, m_\Lambda)=D_\phi^{+}(\tilde\Lambda)$. As shown in \cite[Theorem 2]{ortega-seip-98}, $\tilde\Lambda$ must satisfy $D_\phi^{+}(\tilde\Lambda) < \density{}$. Note that in \cite{ortega-seip-98}, the space $\fock{\phi}$ is denoted $\fock{2\phi}$.

	Towards the sufficiency of the density conditions, let $(\Lambda, m_\Lambda)$ be a separated set with multiplicities, satisfying $D_{\phi}^{+}(\Lambda,m_{\Lambda}) < \density{}$. Hence, for some $\varepsilon > 0$ we have $D_{\phi}^{+}(\Lambda,m_{\Lambda}) < \density{} - \varepsilon$. Thus, for a certain $R > 0$ and for all $z\in\bC$,
	$$\frac{\numberofpoints{z}{r}{\Lambda}{\Lambda}}{|B(z,R)|} < \left(\ddensity{}-\varepsilon\right) \dashint_{B(z,R)} \Delta\phi\, \dA,$$
	where the bar in the integral denotes an average (division by the total measure).
	Let $\chi_{R}=\left(\pi R^{2}\right)^{-1} 1_{B(0, R)}$, $\tilde\phi=\phi\star\chi_R$, and $\nu_{\Lambda} :=\sum_{\lambda \in \Lambda} m_\Lambda (\lambda) \cdot  \delta_{\lambda}$. Then
\begin{align*}
	\pi\frac{\numberofpoints{z}{r}{\Lambda}{\Lambda}}{|B(z,R)|}&=\pi \nu_{\Lambda} \star \chi_R (z) < \left(1-\pi \varepsilon\right) \dashint_{B(z,R)} \Delta\phi\, \dA
	\\
	&= \left(1-\pi \varepsilon\right) \Delta\tilde\phi(z)
	< \Delta\tilde\phi(z)-\tfrac{\pi\varepsilon}{2}\inf_{w \in \mathbb{C}} \Delta\phi(w).
	\end{align*}
	We wish to apply Proposition \ref{prop:interpolation_sufficiency}, with the roles of the weights $\phi$ and $\tilde\phi$  interchanged. The density condition \eqref{eq_bbb} is verified for $\tilde\phi$, because of the previous equation. In addition, $\Delta \tilde\phi$ is continuous, 
	and	$\Delta \tilde\phi$ is bounded above and below with the same constants that bound $\Delta \phi$. Moreover, we claim that
	\begin{align}\label{eq_ccc}
	\norm{\partial^{j} (\phi - \tilde\phi)}_\infty < \infty, \qquad 0 \leq j \leq n_\Lambda.
	\end{align}
	To see this, let $z_0 \in \mathbb{C}$ and let
	\begin{align}\label{eq_ddd}
	\phi(w) = h(w) + G[\Delta \phi](w), \qquad w \in B(z_0,R+1),
	\end{align}
	be the Riesz decomposition of $\phi$ on $B(z_0,R+1)$, as given by Theorem \ref{thm:Riesz_decomp}.
	Let $z \in B(z_0,1)$. We average \eqref{eq_ddd} over $B(z,R)$ and use the mean value property for harmonic functions to conclude
	\begin{align*}
	\tilde\phi(z) = h(z) + G[\Delta \phi]*\chi_{R}(z),
	\end{align*}
	and, therefore,
	\begin{align*}
	\phi(z)-\tilde\phi(z) = G[\Delta \phi](z) - G[\Delta \phi]*\chi_{R}(z), \qquad z \in B_1(z_0).
	\end{align*}
	The desired number of derivatives of each term on the right-hand side of the previous equation is bounded
	independently of $z_0$ by Lemma \ref{lemma:bounds_potentials}, \eqref{eq_fff} (the bound does depend on $R$). Hence, \eqref{eq_ccc} follows.
	We can therefore apply Proposition \ref{prop:interpolation_sufficiency}, with the roles of the weights $\phi$ and $\tilde\phi$ interchanged, and conclude that $\left(\Lambda, m_\Lambda\right)$ is an interpolating set for $\fock{\phi}$ (indeed, $\fock{\phi}=\fock{\tilde\phi}$,
	and the interpolation property is considered with respect to $\bar\partial^*_\phi$.)
\end{proof}

\begin{proof}[Proof of Theorem B]
	Suppose that $(\Lambda, m_{\Lambda})$ is sampling for \fock{\phi}. Then, for some constant $B>0$,
	$$
	\sum_{\lambda \in \Lambda} \left|f(\lambda)\right|^{2}  e^{-\phi(\lambda)} \leq
	\sum_{\lambda \in \Lambda}    \sum_{k=0}^{m_{\Lambda}(\lambda)-1} \left|\dbarstar{f(\lambda)}{k}\right|^{2} e^{-\phi(\lambda)} \leq B\|f\|_{\phi}^{2}.
	$$
The upper sampling bound without multiplicities already implies that $\Lambda$ is relatively separated; see, e.g., \cite[Prop. 1]{ortega-seip-98}. The existence of a separated set $\Lambda'\subseteq\Lambda$ which is also sampling for $\fock{\phi}$ follows from standard arguments (see for instance \cite[Theorem 3.7]{MR2729877}). Without loss of generality we assume that $\Lambda$ is already separated.
	By repeated application of Proposition \ref{prop:ss_decreasing}, there is a sampling set $\tilde\Lambda\subseteq\bC$ without multiplicities, such that $D_\phi^{-}(\Lambda, m_\Lambda)=D_\phi^{-}(\tilde\Lambda)$. As shown in \cite[Theorem 1]{ortega-seip-98}, $\tilde\Lambda$ must satisfy $D_\phi^{-}(\tilde\Lambda) > \density{}$.

Conversely, suppose that $(\Lambda, m_\Lambda)$ is a relatively separated set with multiplicities containing a separated subset $\Lambda^\prime$, satisfying $D_{\phi}^{-}\left(\Lambda^{\prime}, \restr{m_\Lambda}{\Lambda^\prime}\right)>\density{}$. The upper sampling bound follows from the fact that $\Lambda$ is relatively separated and Lemma \ref{lemma:compare_dbarstar_normf}. For the lower sampling bound we assume without loss of generality that $\Lambda=\Lambda^\prime$. Proceeding as in the proof of Theorem A, we consider $\tilde\phi = \phi \star \chi_R$, for some large $R>0$ and conclude from Proposition \ref{prop:sampling-sufficiency},  that $\left(\Lambda, m_\Lambda\right)$ is a sampling set for $\fock{\tilde\phi}$. By \eqref{eq_ccc}, $\fock{\phi}$ and $\fock{\tilde\phi}$ contain the same functions and have equivalent norms. Moreover, for each $f\in\fock{\phi}$,
	we combine \eqref{eq:relation_operators_phi_phi_average} and \eqref{eq_ccc} to conclude that
$$\norm{f}_{\phi}^2 \lesssim \norm{f}_{\tilde\phi}^2 \lesssim \sum_{\lambda \in \Lambda}  \sum_{k=0}^{m_{\Lambda}(\lambda)-1} \left|\dbarstarr{f(\lambda)}{k}{\tilde\phi}\right|^{2} e^{-\tilde\phi(\lambda)} \lesssim \sum_{\lambda \in \Lambda}\sum_{k=0}^{m_{\Lambda}(\lambda)-1} \left|\dbarstarr{f(\lambda)}{k}{\phi}\right|^{2} e^{-\phi(\lambda)}.$$
This gives the desired bound.
\end{proof}


\appendix\section{Bell Polynomials and Fa\`{a} di Bruno's formula}\label{sec:Bell}

The Bell polynomials appear naturally when calculating the $n$-th derivative of a composite function. For a deeper discussion of this topic we refer the reader to \cite[pp. 95--98]{MR1492593} and for the proofs \cite{bell1934}.
\subsection{Bell polynomials}
Let $n, k \in \mathbb{N}$ such that $k \leq n$. The partial exponential Bell polynomials are a collection of polynomials given by
\begin{equation*}
B_{n, k}\left(x_{1}, x_{2}, \ldots, x_{n-k+1}\right)=\sum \frac{n !}{m_{1} ! m_{2} ! \cdots m_{n-k+1} !}\left(\frac{x_{1}}{1 !}\right)^{m_{1}}\left(\frac{x_{2}}{2 !}\right)^{m_{2}} \cdots\left(\frac{x_{n-k+1}}{(n-k+1) !}\right)^{m_{n-k+1}}
\end{equation*}
where the sum is taken over all sequences $m_{1}, m_{2}, m_{3}, \ldots, m_{n-k+1}$ of non-negative integers such that the following conditions are satisfied:
\begin{equation*}
\begin{array}{l}{m_{1}+m_{2}+\cdots+m_{n-k+1}=k} \\ {m_{1}+2 m_{2}+3 m_{3}+\cdots+(n-k+1) m_{n-k+1}=n}\end{array} .
\end{equation*}
If $n \geq 1$, the sum
\begin{equation}\label{eq:def_bell_polynomial}
B_{n}\left(x_{1}, \ldots, x_{n}\right)=\sum_{k=1}^{n} B_{n, k}\left(x_{1}, x_{2}, \ldots, x_{n-k+1}\right)
\end{equation}
is called the $n-$th complete exponential Bell polynomial.
If $n=0$, then
$$
B_{0} := 1.
$$

\begin{remark}\label{remark:constant_term}
	The partial Bell polynomials $B_{n,k}$ are homogeneous polynomials of degree k, therefore it follows from \eqref{eq:def_bell_polynomial} that the $n-$th complete Bell polynomial $B_n$ has no constant term as long as $n \geq 1$.
\end{remark}

The complete Bell polynomials satisfy the binomial type relation:
\begin{equation}\label{eq:bell_decomposition}
B_{n}\left(x_{1}+y_{1}, \ldots, x_{n}+y_{n}\right)=\sum_{i=0}^{n}\binom{n}{i} B_{n-i}\left(x_{1}, \ldots, x_{n-i}\right) B_{i}\left(y_{1}, \ldots, y_{i}\right).
\end{equation}

\subsection{Chain rule for higher derivatives}
 Suppose that $f$ and $r$ are $n$-times differentiable functions, then
\begin{equation*}
\frac{d^{n}}{d x^{n}} f(r(x))=\sum_{k=1}^{n} f^{(k)}(r(x)) \cdot B_{n, k}\left(r^{\prime}(x), r^{\prime \prime}(x), \ldots, r^{(n-k+1)}(x)\right).
\end{equation*}

In what follows, we use the abbreviation
\begin{align}
B_{n, k}^{r}(x) &:= B_{n, k}\left(r^{\prime}(x), r^{\prime \prime}(x), \ldots, r^{(n-k+1)}(x)\right), \label{not:bell_partial}\\
B_{n}^{r}(x) &:= B_{n}\left(r^{\prime}(x), r^{\prime \prime}(x), \ldots, r^{(n)}(x)\right)\label{not:bell_complete}.
\end{align}

\begin{remark}\label{rem:binomial_sum_polynomial}
	Suppose that $f$ and $r$ are functions both differentiable $n$ times. With the notation introduced in \eqref{not:bell_partial} and \eqref{not:bell_complete}, \eqref{eq:bell_decomposition} can be written as
	\begin{equation}\label{eq_mmma}
	B_{n}^{f+g}=\sum_{i=0}^{n}\binom{n}{i} B_{n-i}^{f} B_{i}^g.
	\end{equation}
\end{remark}

One special case of $\frac{d^{n}}{d x^{n}} f(r(x))$ is when $f(t)=e^t$, obtaining:
\begin{align}\label{eq:exponential_derivative}
\frac{d^{n}}{d x^{n}} e^{r(x)}&=\sum_{k=1}^{n} e^{r(x)} \cdot B_{n, k}^{r}(x) =  e^{r(x)} \sum_{k=1}^{n}  B_{n, k}^{r}(x) =  e^{r(x)} B_{n}^{r}(x).
\end{align}

\subsection{Proof of Equation
	\ref{eq:explicit_formula_interpolation_coeficients}
}\label{sec:equation_coef}

We want to prove that the coefficients $k_j$ defined by
\eqref{eq:explicit_formula_interpolation_coeficients}
for $j\in\{0,\dots,m_\Lambda(\lambda)-1\}$
solve the recursive equation \eqref{eq:interpolation_equivalent_condition}.
We proceed by induction on $j$.
The case $j=0$ is clear. Suppose that for every $0\leq l< j$,
	\begin{equation*} k_l = (-1)^l  c_{l} -  {\sum_{m=0}^{l-1} \binom{l}{m}\ k_{m}
}\ B_{l-m}^{G_{\lambda}-\phi}(\lambda),\end{equation*}
or, equivalently,
	\begin{equation*} (-1)^l  c_{l} =   {\sum_{m=0}^{l} \binom{l}{m}\ k_{m}
}\ B_{l-m}^{G_{\lambda}-\phi}(\lambda).\end{equation*}
We use the inductive hypothesis, together with fact that $B_{0}^{\phi-G_{\lambda}} \equiv 1 \equiv B_{0}^{G_\lambda-\phi}$,
and \eqref{eq_mmma} to compute:
\begin{align}%
k_j ={}& (-1)^j c_{j} +  {\sum_{l=0}^{j-1} (-1)^l \binom{j}{l}\ c_{l}
}\ B_{j-l}^{\phi-G_\lambda}(\lambda) \nonumber \\
={}& (-1)^j c_{j}  + {\sum_{l=0}^{j-1} \binom{j}{l}\ \left(
	{\sum_{m=0}^{l} \binom{l}{m}\  k_{m}
	}\ B_{l-m}^{G_\lambda-\phi}(\lambda) \right)
}\ B_{j-l}^{\phi - G_\lambda}(\lambda) \nonumber  \\
\begin{split}
={}& (-1)^j c_{j}  +  {\sum_{l=0}^{j} \binom{j}{l}\ \left(
	{\sum_{m=0}^{l} \binom{l}{m}\  k_{m}
	}\ B_{l-m}^{G_\lambda-\phi}(\lambda) \right)
}\ B_{j-l}^{\phi-G_\lambda}(\lambda)\ - \nonumber  \\
& 
\qquad\qquad\quad{\sum_{m=0}^{j} \binom{j}{m}\ k_{m}
}\ B_{j-m}^{G_\lambda-\phi}(\lambda) 
\end{split} \nonumber \\
={}& (-1)^j c_{j}  +  {\sum_{m=0}^{j} \binom{j}{m}\ k_{m}\ \left(
	{\sum_{l=m}^j \binom{j-m}{j-l}\
		B_{(j-m)-(j-l)}^{G_\lambda-\phi}(\lambda)\ B_{j-l}^{\phi - G_\lambda}(\lambda) }\right)\
} -  \nonumber \\
& \qquad\qquad\quad{\sum_{m=0}^{j} \binom{j}{m}\  k_{m}
}\ B_{j-m}^{G_\lambda-\phi}(\lambda) \nonumber \\
={}& (-1)^j c_{j} +  {\sum_{m=0}^{j} \binom{j}{m}\ k_{m}\ \left(
	{\sum_{l=0}^{j-m} \binom{j-m}{l}\
		B_{(j-m)-l}^{G_\lambda-\phi}(\lambda)\ B_{l}^{\phi - G_\lambda}(\lambda) }\right)\
} -  \nonumber \\
& \qquad\qquad\quad {\sum_{m=0}^{j} \binom{j}{m}\  k_{m}
}\ B_{j-m}^{G_\lambda-\phi}(\lambda) \nonumber \\
={}& (-1)^j c_{j} +  k_j + {\sum_{m=0}^{j-1} \binom{j}{m}\ k_{m}\ 
		B_{j-m}^{0} (\lambda)} \  -  {\sum_{m=0}^{j-1} \binom{j}{m}\  k_{m}
}\ B_{j-m}^{G_\lambda-\phi}(\lambda) \ -\  k_j.
\nonumber
\end{align}
Finally, by Remark \eqref{remark:constant_term}, $B_{j-m}^{0}=0$
if $j-m >0$ and we obtain:
\begin{align*}
k_j = (-1)^j c_j -  {\sum_{m=0}^{j-1} \binom{j}{m}\  k_{m}
}\ B_{j-m}^{G_\lambda-\phi}(\lambda),
\end{align*}
as desired.

\end{document}